\documentclass[a4paper,11pt,notitlepage,twoside]{artikel}

\title{Classification of type $\III$ Bernoulli crossed products}
\author{Stefaan Vaes\thanks{KU~Leuven, Department of Mathematics, Leuven (Belgium), stefaan.vaes@wis.kuleuven.be \newline
    Funded by ERC Consolidator Grant 614195 from the European Research Council under the European Union's Seventh Framework Programme, and by Research Programme G.0639.11 of the Research Foundation~-- Flanders (FWO).}\and Peter Verraedt\thanks{KU~Leuven, Department of Mathematics, Leuven (Belgium), peter.verraedt@wis.kuleuven.be\newline Supported by a Ph.~D.~fellowship of the Research Foundation -- Flanders (FWO).}}

\date{}

\newcommand{\I}{\operatorname{I}}
\newcommand{\II}{\operatorname{I{\kern -0.2ex}I}}
\newcommand{\III}{\operatorname{I{\kern -0.2ex}I{\kern -0.2ex}I}}
\newcommand{\Sd}{\operatorname{Sd}}
\newcommand{\Inn}{\operatorname{Inn}}
\newcommand{\C}{\mathbb{C}}
\newcommand{\F}{\mathbb{F}}

\newcommand{\actson}{\curvearrowright}

\newcommand{\Aut}{\operatorname{Aut}}
\newcommand{\Out}{\operatorname{Out}}
\newcommand{\T}{\mathbb{T}}
\newcommand{\Z}{\mathbb{Z}}
\newcommand{\cF}{\mathcal{F}}

\newcommand{\id}{\mathord{\operatorname{id}}}
\newcommand{\si}{\sigma}
\newcommand{\cE}{\mathcal{E}}
\newcommand{\recht}{\rightarrow}
\newcommand{\cU}{\mathcal{U}}
\newcommand{\vphi}{\varphi}

\newcommand{\al}{\alpha}
\newcommand{\Tr}{\operatorname{Tr}}
\newcommand{\ovt}{\mathbin{\overline{\otimes}}}
\newcommand{\ootimes}{\ovt}

\newcommand{\om}{\omega}
\newcommand{\cP}{\mathcal{P}}
\newcommand{\cZ}{\mathcal{Z}}
\newcommand{\cK}{\mathcal{K}}

\newcommand{\ot}{\otimes}
\newcommand{\cL}{\mathcal{L}}
\newcommand{\Ad}{\operatorname{Ad}}
\newcommand{\dpr}{^{\prime\prime}}

\newcommand{\Mtil}{\widetilde{M}}

\newcommand{\Stab}{\operatorname{Stab}}

\newcommand{\cN}{\mathcal{N}}
\newcommand{\cS}{\mathcal{S}}
\newcommand{\Om}{\Omega}
\newcommand{\cC}{\mathcal{C}}

\newcommand{\cT}{\mathcal{T}}
\newcommand{\Ptil}{\widetilde{P}}
\newcommand{\Hbar}{\overline{H}}
\newcommand{\vphih}{\widehat{\varphi}}
\newcommand{\tauh}{\widehat{\tau}}
\newcommand{\HS}{\mathcal{HS}}
\newcommand{\altil}{\widetilde{\alpha}}
\newcommand{\psitil}{\widetilde{\psi}}
\newcommand{\RR}{\mathbb{R}}
\newcommand{\Psitil}{\widetilde{\Psi}}
\newcommand{\ZZ}{\mathbb{Z}}
\newcommand{\M}{M}
\newcommand{\Q}{Q}
\newcommand{\N}{N}
\newcommand{\inp}[2]{\langle#1, #2\rangle}
\newcommand{\fdot}{\mathop{\cdot}}
\newcommand{\bigootimes}{\mathbin{\overline\bigotimes}}
\newcommand{\pntspectrum}[1]{\ensuremath{\mathop{\text{point spectrum}} \Delta_{#1}}}

\renewcommand{\P}{P}
\renewcommand{\hat}{\widehat}
\renewcommand{\tilde}{\widetilde}
\renewcommand{\mod}{\operatorname{mod}}

\begin{document}
\maketitle
\setlength{\abovedisplayskip}{6pt}
\setlength{\belowdisplayskip}{6pt}

\begin{abstract}
\noindent Crossed products with noncommutative Bernoulli actions were introduced by Connes as the first examples of full factors of type $\III$. This article provides a complete classification of the factors $(\P,\phi)^{\F_n} \rtimes \F_n$, where $\F_n$ is the free group and $\P$ is an amenable factor with an almost periodic state $\phi$. We show that these factors are completely classified by the rank $n$ of the free group $\F_n$ and Connes's $\Sd$-invariant. We prove similar results for free product groups, as well as for classes of generalized Bernoulli actions.
\end{abstract}

\section{Introduction}

During the last decade, major breakthroughs in the classification of type $\II_1$ factors, such as group measure space constructions $L^\infty(X) \rtimes \Gamma$ associated with probability measure preserving (pmp) free ergodic actions, have been achieved.
In \cite{popa-vaes;unique-cartan-decomposition-factors-free}, it was proved that the crossed product $L^\infty(X) \rtimes \F_n$ with an \emph{arbitrary} free ergodic pmp action of the free group $\F_n$ has $L^\infty(X)$ as its unique Cartan subalgebra up to unitary conjugacy. In combination with work of Gaboriau \cite{gaboriau;invariants-relations-equivalence-groupes} and Bowen \cite{bowen;orbit-equivalence-coinduced-actions-free-products}, this yielded the classification of the Bernoulli crossed products $L^\infty(X_0^{\F_n}) \rtimes \F_n$ of the free groups, showing that the rank of $\F_n$ is a complete invariant. In particular, the specific choice of the base probability space $(X_0,\mu_0)$ plays no role.

Classifying a family of von Neumann algebras involves two steps: distinguishing between the nonisomorphic ones and proving that the others are isomorphic. For the first step, a lot of results were obtained in the framework of Popa's deformation/rigidity theory, and in the type~$\III$ case, these must be combined with the modular theory of Connes and Takesaki and the corresponding invariants for type $\III$ factors. The main source for the second step is of a different nature and ultimately makes use of Connes's discovery that in the amenable case, ``everything is isomorphic''~: classification of amenable factors \cite{connes;classification-injective-factors,haagerup;uniqueness-injective-factor-type-III}, orbit equivalence of free ergodic pmp actions of amenable groups \cite{ornstein-weiss;ergodic-theory-amenable-group-actions}, and cocycle conjugacy of outer actions of amenable groups on the hyperfinite $\II_1$ factor \cite{ocneanu;actions-amenable-groups}. Only in rare circumstances, both steps can be fully achieved simultaneously.

In this article, we provide a complete classification of the noncommutative Bernoulli crossed products of the free groups. These factors were introduced by Connes \cite{connes;almost-periodic-states} as the first examples of full factors of type $\III$; we briefly present their construction.
Let $(\P,\phi)$ be any von Neumann algebra equipped with a normal faithful state $\phi$, and let $\Lambda$ be a countably infinite group. Then take the infinite tensor product $\P^\Lambda = \bigootimes_{g \in \Lambda} \P$ indexed by $\Lambda$, with respect to the state $\phi$. The group $\Lambda$ acts on $\P^\Lambda$ by shifting the tensor factors. This action is called a \emph{noncommutative Bernoulli action}. If $\phi$ is not a trace, $N$ is a type $\III$ factor. The factor $\P^\Lambda \rtimes \Lambda$ is called a \emph{Bernoulli crossed product}. We always assume that the base algebra $(\P,\phi)$ is amenable and carries an almost periodic state, and we denote by $\Gamma(\P,\phi)$ the subgroup of $\mathbb{R}^+_0$ generated by the point spectrum of the modular operator $\Delta_\phi$.

A first and obvious invariant for the Bernoulli crossed products $P^\Lambda \rtimes \Lambda$ is amenability. If $\Lambda$ is infinite amenable, then by \cite{connes;classification-injective-factors,haagerup;uniqueness-injective-factor-type-III}, the Bernoulli crossed product $\P^\Lambda \rtimes \Lambda$ is completely determined by its type~: $\II_1$ if $\phi$ is a trace, $\III_\lambda$ for $\lambda \in (0,1)$ if $\Gamma(\P,\phi) =\{\lambda^n \mid n \in \mathbb{Z}\}$, and $\III_1$ if $\Gamma(\P,\phi)$ is dense in $\mathbb{R}^+$.

If $\Lambda$ is nonamenable, the Bernoulli crossed product $\P^{\Lambda}\rtimes \Lambda$ is a full factor and Connes' $\Sd$-invariant equals $\Gamma(\P,\phi)$; see \cite{connes;almost-periodic-states} and \cref{lemma:fullfactor} below. Hence, compared to the amenable case, more information on the state $\phi$ is preserved.
Using Popa's deformation/rigidity theory, and in particular the results of \cite{popa-vaes;unique-cartan-decomposition-factors-hyperbolic}, we show that the group $\Lambda$ is also an invariant among all the Bernoulli crossed products $\P^{\Lambda} \rtimes \Lambda$ with $\Lambda$ belonging to a large class $\cC$ of countable groups including all nonelementary hyperbolic groups (see \cref{definition:class-C}). Conversely, combining Ocneanu's classification of outer actions of amenable groups on amenable factors \cite{ocneanu;actions-amenable-groups} and Bowen's co-induction argument \cite{bowen;orbit-equivalence-coinduced-actions-free-products}, we prove that the factors $(P_i,\phi_i)^\Lambda \rtimes \Lambda$, $i=0,1$, are isomorphic whenever $\Lambda = \Sigma \star \Upsilon$ is a free product with $\Sigma$ infinite amenable and $\Gamma(P_0,\phi_0) = \Gamma(P_1,\phi_1)$.

Altogether, this gives the first main result of the article.

\begin{thmstar}\label{thm.A}
The set of factors
\begin{align*}
\bigl\{(P,\phi)^\Lambda \rtimes \Lambda \bigm| &\;\;\text{$P$ a nontrivial amenable factor with normal faithful almost periodic state $\phi$,}\\
&\;\;\text{and $\Lambda = \Sigma \star \Upsilon$ a free product of an infinite amenable group $\Sigma$ and}\\ &\;\;\text{a nontrivial countable group $\Upsilon$}\;\bigr\}
\end{align*}
is exactly classified, up to isomorphism, by $\Gamma(P,\phi) \subset \RR^+_0$ and the isomorphism class of $\Lambda$.
\end{thmstar}

When we replace the free product group $\Lambda = \Sigma \star \Upsilon$ by other classes of groups, such as the direct product of two free groups, a more rigid behavior appears and the factor $(P,\phi)^\Lambda \rtimes \Lambda$ retains more information on the state $\phi$. We obtain the following result. The class $\cC$ of countable groups is introduced in \cref{definition:class-C} and by \cref{remark:many-in-class-C}, it contains all nonelementary hyperbolic groups, as well as all nontrivial free product groups.

\begin{thmstar}\label{thm.B}
The set of factors
\begin{align*}
\bigl\{(P,\phi)^\Lambda \rtimes \Lambda \bigm| &\;\;\text{$P$ a nontrivial amenable factor with normal faithful almost periodic state $\phi$,}\\
&\;\;\text{and $\Lambda$ a direct product of two icc groups in the class $\cC$}\;\bigr\}
\end{align*}
is exactly classified, up to isomorphism, by the group $\Lambda$ and the action $\Lambda \actson (P,\phi)^\Lambda$ up to a state preserving conjugacy of the action.
\end{thmstar}

We believe that in general, the existence of a state preserving conjugacy between the actions $\Lambda \actson (P_i,\phi_i)^\Lambda$ is a strictly stronger condition than the equality of the groups $\Gamma(P_i,\phi_i)$ that appears in \cref{thm.A}. This can be best illustrated when $\Lambda = \ZZ$ and $(P,\phi)$ is a factor of type $\I$ with a normal faithful state. According to the work of Connes and St{\o}rmer \cite{connes-stormer;entropy-automorphisms} and Connes, Narnhofer and Thirring \cite{connes-narnhofer-thirring;dynamical-entropy}, apart from the subgroup $\Gamma(P,\phi) \subset \RR^+_0$, also the entropy $H(\phi)$ is a state preserving conjugacy invariant of $\Lambda \actson (P,\phi)^\Lambda$.
Here the entropy of a normal faithful state $\phi$ on a type $\I$ factor $P$ is given by
\begin{align*}
 H(\phi) = - \sum_k \phi(p_k) \log\phi(p_k),
\end{align*}
where the $p_k$ form a maximal orthogonal family of minimal projections in the centralizer $P_\phi$.

It is highly plausible that for large classes of infinite groups $\Lambda$ and type $\I$ factors $P$, the entropy $H(\phi)$ is a state preserving conjugacy invariant of $\Lambda \actson (P,\phi)^\Lambda$. Some evidence for this can be found in the commutative case, as Bowen \cite{bowen;measure-conjugacy-invariant-free-group-actions,bowen;measure-conjugacy-invariants-sofic-groups} showed that the entropy $\mu_0$ of the base probability space $(X_0,\mu_0)$ is a conjugacy invariant for Bernoulli actions $\Lambda \curvearrowright (X_0,\mu_0)^{\Lambda}$ of sofic groups. Furthermore, for every countably infinite group $\Lambda$, two Bernoulli actions $\Lambda \curvearrowright X_0^\Lambda$ and $\Lambda \curvearrowright Y_0^\Lambda$ were shown to be conjugate once $(X,\mu_0), (Y,\nu_0)$ are both not two-atomic and $H(\mu_0) = H(\nu_0)$, see \cite{ornstein;bernoulli-shifts-same-entropy,ornstein;bernoulli-shifts-infinite-entropy,bowen;every-group-almost-ornstein}.

One can thus even speculate that for large classes of countable groups $\Lambda$, the existence of a state preserving conjugacy between two Bernoulli actions $\Lambda \actson (P_i,\phi_i)^\Lambda$, $i=0,1$, with $P_i$ amenable factors and $\phi_i$ normal faithful almost periodic states, is equivalent with the equality of both $\Gamma(P_0,\phi_0) = \Gamma(P_1,\phi_1)$ and $H(\phi_0) = H(\phi_1)$. To prove such a statement, there is a two-fold open problem to be solved. Is Connes-Narnhofer-Thirring entropy a conjugacy invariant for noncommutative Bernoulli actions of, say, sofic groups? In other words, is it possible to develop a noncommutative version of Bowen's sofic entropy? Secondly, do equal entropy and equal point spectrum imply conjugacy of the actions? In other words, is there a noncommutative Ornstein theorem? This is a `classical' problem that is open for the group $\ZZ$. But once it is solved for $\ZZ$, by co-induction, the same result follows for all groups containing $\ZZ$.

In contrast with the situation above, a conjugacy between \emph{two-sided} Bernoulli actions $\Lambda \times \Lambda \curvearrowright (\P_0,\phi_0)^\Lambda$ and $\Lambda \times \Lambda \curvearrowright (\P_1,\phi_1)^\Lambda$ for an icc group $\Lambda$ yields more information, as it provides an isomorphism between the base algebras. Thus, for two-sided Bernoulli actions, we get the following optimal result. We say that a sequence of elements $g_n$ in a group $\Lambda$ is central if for every $h \in \Lambda$, we have that $g_n h = h g_n$ eventually. We call such a sequence trivial if $g_n = e$ eventually.

\begin{thmstar}\label{thm.C}
The set of factors
\begin{align*}
\bigl\{(P,\phi)^\Lambda \rtimes (\Lambda \times \Lambda) \bigm| &\;\text{$P$ a nontrivial amenable factor, $\phi$ a normal faithful almost periodic state,}\\
&\;\text{and $\Lambda$ an icc group in the class $\cC$ without nontrivial central sequences}\;\bigr\}
\end{align*}
is exactly classified, up to isomorphism, by the group $\Lambda$ and the pair $(P,\phi)$ up to state preserving isomorphism.
\end{thmstar}

As we see below, the assumption on central sequences in \cref{thm.C} exactly ensures that the crossed product $P^\Lambda \rtimes (\Lambda \times \Lambda)$ is full.

The proofs of the above classification theorems for almost periodic type $\III$ factors all make use of the modular theory of Tomita-Takesaki-Connes and the so-called discrete decomposition of an almost periodic factor, see \cref{section:almostperiodicweights}. In \cref{lemma:cocycleconjugatecore}, we provide the technical lemma that relates cocycle conjugacy of state preserving actions on almost periodic factors $(P,\phi)$ with cocycle conjugacy of the associated actions on the discrete decomposition of $(P,\phi)$.

To prove theorems \ref{thm.A}, \ref{thm.B} and \ref{thm.C}, there is an `isomorphism part' in which we prove that the factors are isomorphic when the invariants are the same, and a `non-isomorphism part' in which we recover the invariants from the factor. The isomorphism part is trivial in theorems \ref{thm.B} and \ref{thm.C}, but is nontrivial in \cref{thm.A}. As explained above, in \cref{section:isomorphism}, we establish the isomorphism part first for Bernoulli actions of infinite amenable groups $\Sigma$ using \cite{ocneanu;actions-amenable-groups} and then co-induce to $\Sigma \star \Upsilon$ by a noncommutative version of \cite{bowen;orbit-equivalence-coinduced-actions-free-products}.

To recover the group $\Lambda$ as an invariant of an almost periodic (generalized) Bernoulli crossed product $(P,\phi)^I \rtimes \Lambda$, we first pass to the discrete decomposition and then use the main results of \cite{popa-vaes;unique-cartan-decomposition-factors-free,popa-vaes;unique-cartan-decomposition-factors-hyperbolic,ioana;cartan-subalgebras-amalgamated-free-product-factors} providing unique crossed product decomposition theorems for factors of the form $R \rtimes \Lambda$, where $\Lambda \actson R$ is an outer action on the hyperfinite $\II_1$ factor $R$ of a group $\Lambda$ that is a free group, or a hyperbolic group, or a free product group, etc. In \cref{section:regular-subalgebras}, we introduce the class $\cC$ of countable groups $\Lambda$ for which such a unique crossed product decomposition theorem holds. We prove that this class $\cC$ is stable under taking extensions and stable under commensurability. \Cref{thm.non-iso-class-C-type-III} provides a general result recovering the cocycle conjugacy class of $\Lambda \actson (P,\phi)$ from the crossed product $P \rtimes \Lambda$, provided $\Lambda$ belongs to the class $\cC$, $P$ is amenable, $\phi$ is almost periodic and $P \rtimes \Lambda$ is a full factor.

Finally, to go from cocycle conjugacy to conjugacy for actions of the form $\Lambda \actson (P,\phi)^I$, we must prove that all $1$-cocycles for the action are trivial. In \cite{popa;rigidity-non-commutative-bernoulli}, Popa proved such a cocycle superrigidity theorem for noncommutative Bernoulli actions of property (T) groups (and more generally, $w$-rigid groups). Later, in  \cite{popa;superrigidity-malleable-actions-spectral-gap}, Popa also established cocycle superrigidity for commutative Bernoulli actions $\Lambda \actson (X,\mu)^\Lambda$ of nonamenable direct product groups $\Lambda$, by introducing his spectral gap methods in deformation/rigidity theory. As an appendix, we adapt the proof of \cite{popa;superrigidity-malleable-actions-spectral-gap} to the noncommutative setting and prove a cocycle superrigidity theorem for Connes-St{\o}rmer Bernoulli actions of nonamenable direct product groups. This then concludes the `non-isomorphism part' of theorems \ref{thm.B} and \ref{thm.C}.

\section{Preliminaries and Notations}

All mentioned von Neumann algebras are assumed to have separable predual.

\subsection{Cocycle actions}

Let $\M_1$ and $\M_2$ be von Neumann algebras equipped with normal semifinite faithful (n.s.f.) weights $\varphi_1$ and $\varphi_2$ respectively. A homomorphism $\alpha : \M_1\to \M_2$ is called \emph{weight scaling} if $\varphi_2 \circ \alpha = \lambda \varphi_1$ for some $\lambda \in \mathbb{R}_0^+$. The number $\lambda$ is called the \emph{modulus} of $\alpha$, and will be denoted as $\mod_{\varphi_1,\varphi_2} \alpha$.
We denote by $\Aut\M$ the group of automorphisms of $\M$, and by $\Aut(\M,\varphi)$ the subgroup of weight scaling automorphisms of $\M$. We equip $\Aut\M$ with the topology where a net $\alpha_n$ of automorphisms of $\M$ converges to $\alpha \in \Aut\M$ if and only if for every $\psi \in \M_\star$, $\| \psi \circ \alpha_n - \psi \circ \alpha\|$ converges to zero. With this topology, the groups $\Aut\M$ and $\Aut(\M,\varphi)$ become Polish groups.

Let $G$ be a locally compact group, and $(\M,\varphi)$ a von Neumann algebra with an n.s.f.~weight.
We denote the \emph{centraliser} of $\varphi$ by $\M_\varphi$.
A \emph{cocycle action} of $G$ on $(\M,\varphi)$ is a continuous map
$ \alpha : G \to \Aut(\M,\varphi) : g \mapsto \alpha_g$ and a continuous map $v : G \times G \to \mathcal{U}(\M_\varphi)$ such that:
\begin{align*}
\alpha_e = \text{id}, \qquad \alpha_{g} \circ \alpha_{h} &= \Ad v_{g,h} \circ \alpha_{gh}, \qquad\forall g,h,k\in G, \\
v_{g,h}v_{gh,k} &= \alpha_g (v_{h,k}) v_{g,hk}.
\end{align*}
Such an action is denoted by $G \curvearrowright^{\alpha,v} (\M,\varphi)$. A strongly continuous map $v$ satisfying the above relations is called a 2-\emph{cocycle} for $\alpha$. If $v=1$, $\alpha$ is called an \emph{action}.

Two cocycle actions $G \curvearrowright^{\alpha,v} (\M_1,\varphi_1)$ and $G \curvearrowright^{\beta,w} (\M_2,\varphi_2)$ are \emph{cocycle conjugate through a weight scaling isomorphism} if there exists a strongly continuous map $u : G \to \mathcal{U}( (\M_2)_{\varphi_2})$ and a weight scaling isomorphism $\psi : (\M_1,\varphi_1) \to (\M_2,\varphi_2)$ satisfying
\begin{align}
\psi \circ \alpha_g &= \Ad u_g \circ \beta_g \circ \psi,\qquad\forall g,h \in G, \label{eq:outerconj}\\
\psi(v_{g,h})&=u_g \beta_g(u_h) w_{g,h} u^\star_{gh}.
\end{align}
If moreover $\mod_{\varphi_1,\varphi_2}\Psi=1$, we say that $(\alpha,v)$ and $(\beta,w)$ are cocycle conjugated \emph{through a weight preserving isomorphism}.
If $\alpha$, $\beta$ are actions and \eqref{eq:outerconj} holds for $u = 1$, we say that $\alpha$ and $\beta$ are \emph{conjugate}.
If $\delta : G_1 \to G_2$ is an isomorphism, we say that the cocycle actions $G_1 \curvearrowright^{\alpha,v} (\M_1,\varphi_1)$ and $G_2 \curvearrowright^{\beta,w} (\M_2,\varphi_2)$ are cocycle conjugate, resp.~conjugate, \emph{modulo $\delta$}, if the actions $\alpha$ and $\beta \circ \delta$ of $G_1$ are cocycle conjugate, resp.~conjugate.

An automorphism $\alpha$ of $\M$ is called \emph{properly outer} if there exists no nonzero element $y \in \M$ such that $y \alpha(x)=xy$ for all $x \in \M$. A cocycle action $\Gamma \curvearrowright^{\alpha,v} (\M,\varphi)$ of a discrete group is called \emph{properly outer} if $\alpha_g$ is properly outer for all $g \in \Gamma, g\not=e$.

Assume now that $(\M,\varphi)$ is a von Neumann algebra with an n.s.f.~weight for which $\M_\varphi$ is a factor. Consider a properly outer cocycle action $\Gamma \curvearrowright^{\alpha,v}(\M,\varphi)$ of a discrete group $\Gamma$, that preserves the weight $\varphi$.
Suppose that $p \in \M_\varphi$ is a nonzero projection with $\varphi(p) < \infty$, and choose partial isometries $w_g \in \M_\varphi$ such that $p = w_gw_g^\star$, $\alpha_g(p) = w_g^\star w_g$, $w_e=p$.
Let $\alpha^p : G \to \Aut(p\M p, \varphi_p)$ be defined by
$ \alpha_g^p(pxp) = w_g\alpha_g(pxp)w_g^\star$, for $x \in \M$, where $\varphi_p(pxp) = \varphi(pxp)/\varphi(p)$. Denote $v_{g,h}^p = w_g \alpha_g(w_h) v_{g,h} w_{gh}^\star \in p\M_\varphi p$, for $g,h\in G$. Then $G \curvearrowright^{\alpha^p,v^p}(p\M p,\varphi_p)$ is a properly outer cocycle action, which does not depend on the choice of the partial isometries $w_g \in \M_\varphi$, up to state preserving cocycle conjugacy. We call this action the \emph{reduced cocycle action of $\alpha$ by $p$}.

Let $\alpha: G \to \Aut(\M,\varphi)$ be an action of a locally compact group on a von Neumann algebra $(\M,\varphi)$ with an n.s.f.~weight $\varphi$.
A \emph{generalized 1-cocycle} for $\alpha$ with \emph{support projection} $p \in \M_\varphi$ is a continuous map $w:G \to \M_\varphi$ such that $w_g \in p M_\varphi \alpha_g(p)$ is a partial isometry with $p=w_g w_g^\star$ and $\alpha_g(p) = w_g^\star w_g$, and
\begin{align*}
 w_{gh} = \Omega(g,h) w_g \alpha_g(w_h) \quad\text{ for all }g,h \in G,
\end{align*}
where $\Omega(g,h)$ is a scalar 2-cocycle.

\subsection{Connes-Takesaki's discrete decomposition for almost periodic factors}\label{section:almostperiodicweights}
An n.s.f.~weight $\varphi$ on a von Neumann algebra $\M$ is called \emph{almost periodic} if the modular operator $\Delta_\varphi$ on $L^2(\M,\varphi)$ is diagonalizable.

Fix an almost periodic n.s.f.~weight $\varphi$ on a von Neumann algebra $\M$. Let $\Gamma \subset \mathbb{R}^+_0$ be the subgroup generated by the point spectrum of $\Delta_\varphi$, and endowed with the discrete topology. Let $\hat{\iota} : \Gamma \hookrightarrow \mathbb{R}^+_0$ denote the inclusion map, and let $G = \hat{\Gamma}$. Consider $\mathbb{R}^+_0$ to be the dual of $\mathbb{R}$ under the pairing $\inp{t}{\mu}=\mu^{\mathbf{i}t}$ for $t \in \mathbb{R}$, $\mu \in \mathbb{R}^+_0$. Then there is a continuous group homomorphism $\iota : \mathbb{R} \to G$ determined by $\inp{\iota(t)}{\gamma}=\inp{t}{\hat\iota(\gamma)}$, for $\gamma\in \Gamma,t\in\mathbb{R}$. Since $\hat\iota$ is injective, the image of $\iota$ is dense in $G$.

There is a unique action $\sigma$ of $G$ on $\M$ such that $\varphi \circ \sigma_g = \varphi$ for $g \in G$
and such that $\sigma_{\iota(t)} = \sigma_t^\varphi$ for all $t \in \mathbb{R}$ (see \cite[Lemma 3.7.3]{connes;classification-facteurs;ecole-normale}).
Here $\sigma^\varphi$ is the modular automorphism group associated to $\varphi$.
Conversely, the existence of a continuous homomorphism from $\mathbb{R}$ to a compact group $G$ with dense image and a group action $\sigma : G \curvearrowright \M$, such that $\sigma_{\iota(t)} = \sigma_t^\varphi$ for all $t \in \mathbb{R}$, implies that $\varphi$ is almost periodic.

For every $\gamma \in \Gamma$, we denote the Arveson-Connes spectral subspace (see also section 2.1 of \cite{connes;classification-facteurs;ecole-normale} and section XI.1 of \cite{takesaki;theory-operator-algebras-II}, whose conventions we follow) of $\gamma$ by
\begin{align*}
\M_{\varphi,\gamma} = \{ x \in \M \mid \forall g \in G : \sigma_g(x) = \inp{g}{\gamma} x\}.
\end{align*}
The linear span of $\bigcup_{\gamma \in \Gamma} \M_{\varphi,\gamma}$ is a strongly dense $\star$-subalgebra of $\M$ (see e.g.~\cite[Lemma 1.2.3]{dykema;crossed-product-decompositions}), and $\M_{\varphi,\gamma} \not=\{0\}$ if and only if $\gamma \in \pntspectrum{\varphi}$.
If $\varphi$ is a state, then $$\M_{\varphi,\gamma} = \{ x \in M \mid \Delta_\varphi \hat{x} = \hat\iota(\gamma) \hat{x}  \},$$ and for all $a \in \M_{\varphi,\gamma}$ and $b \in \M$ we have $\varphi(a b) = \hat\iota(\gamma) \varphi(b a)$.

The crossed product $\M \rtimes_\sigma G$ of $\M$ with the action $\sigma$ is called the \emph{discrete core} of $\M$. We denote by $\pi_\sigma : \M \to \M \rtimes_\sigma G$ the canonical embedding, and by $(\lambda(s))_{s \in G}$ the canonical group of unitaries such that
\begin{align*}
 \pi_\sigma( \sigma_s(x)) = \lambda(s) \pi_\sigma(x) \lambda(s)^\star.
\end{align*}
We denote by $\hat{\sigma}$ the \emph{dual action} of $\Gamma$ on $\M \rtimes_\sigma G$, given by
\begin{align*}
 \hat{\sigma}_\gamma(\pi_\sigma(x)) = \pi_\sigma(x) \text{ for all } x \in \M \quad \text{ and } \quad \hat{\sigma}_\gamma(\lambda(s)) = \overline{\inp{s}{\gamma}} \lambda(s) \text{ for all } s \in G.
\end{align*}
The dynamical system $(\M\rtimes_\sigma G, \Gamma, \hat\sigma)$ is the \emph{discrete decomposition} associated to $\varphi$.
We have by Takesaki duality \cite[Theorem X.2.3]{takesaki;theory-operator-algebras-II} that
\begin{align*}
(\M \rtimes_\sigma G) \rtimes_{\hat\sigma} \Gamma \cong \M \ootimes B(\ell^2(\Gamma))
\end{align*}
by the isomorphism $\Phi : (\M \rtimes_\sigma G) \rtimes_{\hat\sigma} \Gamma \to \M \ootimes B(\ell^2(\Gamma))$ given by
\begin{align}
\Phi\big(\pi_{\hat\sigma} \circ \pi_\sigma(x)\big) &= x \otimes \lambda_\gamma^\star,&&\gamma \in \Gamma, \quad x \in \M_{\varphi,\gamma},\nonumber\\
\Phi\big(\pi_{\hat\sigma} \circ \lambda(s)\big) &= 1 \otimes M_{\inp{s}{\fdot}}^\star, &&s\in G,\label{eq:takesakiduality}\\
\Phi\big(\lambda(\gamma)\big) &= 1 \otimes \lambda_\gamma^\star. && \nonumber
\end{align}
Here $M_{\inp{s}{\fdot}}$ is the multiplication operator $\delta_\gamma \mapsto \inp{s}{\gamma}\delta_\gamma$, while $\lambda_\gamma$ is the translation operator $\delta_{\gamma'} \mapsto \delta_{\gamma \gamma'}$, and $(\lambda(\gamma))_{\gamma \in \Gamma}$ is the canonical group of unitaries in the second crossed product.

There exists an n.s.f.~trace $\Tr_\varphi$ on $\M\rtimes_\sigma G$ such that
\begin{align*}
\Tr_\varphi \circ \hat\sigma_{\gamma} = \hat\iota(\gamma)^{-1} \Tr_\varphi  \text{ for all }\gamma \in \Gamma,
\end{align*}
and such that the dual weight $\tilde\Tr_\varphi$ (see \cite[Definition X.1.16]{takesaki;theory-operator-algebras-II}) on $(\M\rtimes_\sigma G)\rtimes_{\hat\sigma}\Gamma$ corresponds under the duality $\Phi$ to the weight $\varphi \otimes \Tr(M_{\hat{\iota}} \fdot)$ on $\M \ootimes B(\ell^2(\Gamma))$, where $M_{\hat{\iota}}$ is the multiplication operator associated to the unbounded positive function $\hat{\iota}$ on $\Gamma$.

The bidual action $\hat{\hat\sigma}$ of $G$ on $(\M \rtimes_\sigma G) \rtimes_{\hat\sigma} \Gamma$ corresponds under $\Phi$ to the action $s \mapsto \sigma_s \otimes \Ad M_{\inp{s}{\fdot}}$ on $\M \ootimes B(\ell^2(\Gamma))$. It follows that $\M \rtimes_\sigma G$ is a factor if and only if $\M_\varphi$ is a factor (see e.g.~\cite[Proposition 2.12]{dykema;crossed-product-decompositions}).
In this case, the point spectrum of $\Delta_\varphi$ is already a group, and we say that $\M$ has a \emph{factorial discrete decomposition}.
Also remark that the dual weight $\tilde\Tr_\varphi$ is almost periodic, and that the bidual action $\hat{\hat{\sigma}}$ of $G$ is exactly the modular action for $\tilde\Tr_\varphi$.

We need the following probably well known lemma.

\begin{lemma}\label{lemma:outer-on-core}
Let $(M,\vphi)$ be a factor equipped with an almost periodic normal faithful state having a factorial discrete decomposition $M \rtimes_\sigma G$. Let $\al \in \Aut(M,\vphi)$ be a state preserving automorphism and denote by $\altil \in \Aut(M \rtimes_\sigma G)$ its canonical extension satisfying $\altil(\lambda(s)) = \lambda(s)$ for all $s \in G$.

Then, the following statements are equivalent.
\begin{enumerate}[\upshape(i)]
\item $\altil$ is an inner automorphism.
\item The restriction of $\al$ to $M_\vphi$ is inner.
\item There exists a unitary $v \in M_\vphi$ and a group element $s \in G$ such that $\al = \Ad v \circ \sigma_s$.
\end{enumerate}
\end{lemma}
\begin{proof}
(i) $\Rightarrow$ (ii). Assume that $\altil$ is an inner automorphism. Denote by $p \in L(G)$ the minimal projection corresponding to the identity element in $\hat{G}$, so that $\lambda(s) p = p$ for all $s \in G$. Since $\altil(p) = p$, also the restriction of $\altil$ to $p(M \rtimes_\sigma G)p$ is inner. But $p(M \rtimes_\sigma G)p$ can be identified with $M_\vphi$ and we conclude that (ii) holds.

(ii) $\Rightarrow$ (iii). Take $v \in \cU(M_\vphi)$ such that $\al(x) = v x v^\star$ for all $x \in M_\vphi$. Define $\beta \in \Aut(M,\vphi)$ given by $\beta(x) = v^\star \al(x) v$ for all $x \in M$. We have to prove that $\beta = \sigma_s$ for some $s \in G$. Note that $\beta$ is state preserving and that $\beta(x) = x$ for all $x \in M_\vphi$.
As above, consider $\Gamma = \hat{G}$ and fix $\gamma \in \Gamma$. Choose a nonzero partial isometry $w$ in the spectral subspace $M_{\vphi,\gamma}$. Denote $p = w^\star w$ and note that $p \in M_\vphi$. Since $\beta$ is state preserving, we have that $\beta$ commutes with the automorphisms $\sigma_s$, $s \in G$. Therefore, $w^\star \beta(w) \in p M_\vphi p$. Let $a \in p M_\vphi p$ be arbitrary. Then also $w a w^\star \in M_\vphi$ and we get that
$$w^\star \beta(w) \; a = w^\star \beta(wa) = w^\star \beta(waw^\star \, w) = w^\star \, waw^\star \, \beta(w) = a \; w^\star \beta(w) \; .$$
Since $p M_\vphi p$ is a factor, we conclude that $w^\star \beta(w) = s(\gamma) p$ for some scalar $s(\gamma) \in \T$. Since $M_\vphi$ is a factor, the linear span of $M_\vphi w M_\vphi$ equals $M_{\vphi,\gamma}$. It follows that $\beta(a) = s(\gamma) a$ for all $\gamma \in \Gamma$ and $a \in M_{\vphi,\gamma}$. We conclude that $s : \Gamma \recht \T$ is a character. So, $s \in G$ and $\beta = \sigma_s$.

(iii) $\Rightarrow$ (i). If $\al = \Ad v \circ \sigma_s$ for some $v \in \cU(M_\vphi)$ and $s \in G$, then $\altil = \Ad(v \lambda(s))$.
\end{proof}

\subsection{Connes's invariants for type \texorpdfstring{$\III_1$}{III1} factors}

To study factors of type $\III_1$, Connes \cite{connes;almost-periodic-states} introduced the invariants $\Sd$ and $\tau$, which we recall here.
Let $\M$ be a factor, then the \emph{point modular spectrum} is the subset of $\mathbb{R}_0^+$ defined by
\begin{align*}
\Sd(\M) = \bigcap_{\psi \text{ almost periodic weight on }\M} \pntspectrum{\psi}.
\end{align*}

A factor $\M$ is \emph{full} when $\Inn(\M)$ is closed inside $\Aut(\M)$.
If $\M$ is a full factor with separable predual, and if $\varphi$ is a normal faithful almost periodic weight on $\M$, then $\Sd(\M) = \pntspectrum{\varphi}$ if and only if $\M_\varphi$ is a factor (see \cite[Lemma 4.8]{connes;almost-periodic-states}).
An arbitrary factor $\M$ is full if and only if every bounded sequence $(x_n)_{n \in \mathbb{N}}$ in $\M$ such that $\forall \psi \in \M_\star : \| \psi(x_n\fdot)-\psi(\fdot x_n)\| \to 0$, is trivial, i.e.~$x_n - z_n1 \to 0$ $\star$-strongly for some sequence $(z_n)_{n\in\mathbb{N}} \in \mathbb{C}$ (see \cite[Theorem 3.1]{connes;almost-periodic-states}). In particular, if all central sequences in $\M$ are trivial, $\M$ is full (see \cite[Corollary 3.7]{connes;almost-periodic-states}). A bounded sequence $(x_n)_{n\in \mathbb{N}}$ is called \emph{central} if for all $y \in \M$, $x_ny-yx_n \to 0$ $\star$-strongly.

Let $\M$ now be a full factor of type $\III_1$. Denote by $\pi : \Aut \M \to \Out \M$ the canonical projection. Let $\psi$ be any weight on $\M$, and define $\delta : \mathbb{R} \to \Out(\M)$ by putting $\delta(t) = \pi(\sigma^\psi_t)$. This map is independent of the choice of the weight $\psi$, due to Connes's Radon-Nikodym cocycle theorem (see \cite[Theorem VIII.3.3]{takesaki;theory-operator-algebras-II}). The $\tau$ invariant of $\M$, denoted by $\tau(\M)$, is the weakest topology on $\mathbb {R}$ that makes $\delta$ continuous.

\subsection{State preserving actions and the discrete core}\label{section:statepreservingactionsonthecore}

Every state preserving action on a factor $(\P,\phi)$ equipped with an almost periodic state induces an action on the discrete core of $(\P,\phi)$. We prove two elementary lemmas, one giving a criterion that this induced action is outer and one relating cocycle conjugacy of induced actions with cocycle conjugacy of the original action.

Let $(\P,\phi)$ be a factor, with a normal faithful almost periodic state $\phi$. Let $\Lambda$ be a countable group, and assume that we are given a state preserving action $\Lambda \curvearrowright^{\alpha} (\P,\phi)$.
Denote by $\Gamma \subset \mathbb{R}^+_0$ the group generated by the point spectrum of $\Delta_{\phi}$, and let $G = \hat{\Gamma}$.
Consider the crossed product $\N = \P \rtimes_{\sigma} G$ of $\P$ by the modular action $\sigma$ of $G$.
Then $\Gamma$ acts on $\N$ by the dual action $\hat\sigma$, and since the actions $\alpha$ and $\sigma$ commute, $\Lambda$ also acts on $\N$ by $\tilde\alpha$:
\begin{align*}
\tilde\alpha_s\big(\pi_{\sigma}(x)\big) &= \pi_{\sigma}\big(\alpha_s(x)\big),&&s \in \Lambda,\quad x \in \P,\\
\tilde\alpha_s(\lambda(g)) &= \lambda(g),&&g\in G.
\end{align*}
A direct calculation shows that also $\tilde\alpha$ and $\hat\sigma$ commute. We thus find the action
$\beta$ of the countable group $\Lambda \times \Gamma$ on $\N$, given by $\beta_{(s,\gamma)} = \tilde\alpha_s \circ \hat{\sigma}_\gamma$.

\begin{lemma}\label{lemma:outer-action-on-core}
Under the above assumptions, suppose that $(P,\phi)$ has a factorial discrete decomposition $N$. If the group $\Lambda$ has trivial center and the action $\Lambda \actson^\al P$ is outer, then also the action $\altil$ of $\Lambda$ on $P \rtimes_\sigma G$ is outer.
\end{lemma}
\begin{proof}
Denote by $\Lambda_0$ the group of all $s \in \Lambda$ for which $\altil_s$ is inner. Viewing both $\Lambda$ and $G$ as subgroups of the outer automorphism group $\Out(P)$, it follows from \cref{lemma:outer-on-core} that $\Lambda_0 = \Lambda \cap G$. Since $\Lambda$ and $G$ commute inside $\Out(P)$, it follows that $\Lambda_0$ is a subgroup of the center of $\Lambda$. Therefore, $\Lambda_0 = \{e\}$.
\end{proof}

\begin{lemma}\label{lemma:cocycleconjugatecore}
Let $(\P_0,\phi_0)$, $(\P_1,\phi_1)$ be factors with normal faithful almost periodic states and factorial discrete decompositions. Assume that $\Delta_{\phi_0}$ and $\Delta_{\phi_1}$ have the same point spectrum $\Gamma = \hat{G}$.
Let $\Lambda$ be a countable group, with state preserving actions $\Lambda \curvearrowright^{\alpha^i} (\P_i,\phi_i)$.

Assume that $\psi : \P_0 \rtimes_{\sigma^0} G \to \P_1 \rtimes_{\sigma^1} G$ is an isomorphism such that the induced actions $\Lambda \times \Gamma \curvearrowright^{\beta^i} \P_i \rtimes_{\sigma^i} G$ are cocycle conjugate through $\psi$.
Then there exists a projection $p \in (\P_1)_{\phi_1}$ such that the cocycle actions $\alpha^0$ and $(\alpha^1)^p$ are cocycle conjugate through a state preserving isomorphism. If $\mod \psi \in \Gamma$, then we can take $p=1$.
\end{lemma}

\begin{proof}
If $\Gamma = \{1\}$, there is nothing to prove, so we assume that $\Gamma \not=\{1\}$. Write $\N_i = \P_i \rtimes_{\sigma^i} G$ and let $\psi : \N_0 \to \N_1$ be the isomorphism given in the theorem. Write $\kappa = \mod \psi$. Since we may replace $\psi$ by $\psi \circ \hat{\sigma}_\gamma^0$ for any $\gamma \in \Gamma$, we may assume that $\kappa \leq 1$ and that $\kappa = 1$ if the original $\mod \psi$ belonged to $\Gamma$.

Since $\psi$ is a cocycle conjugacy between the actions $\Lambda \times \Gamma \actson N_i$, we can first extend $\psi$ to an isomorphism $\psitil : N_0 \rtimes_{\hat{\sigma}^0} \Gamma \recht N_1 \rtimes_{\hat{\sigma}^1} \Gamma$ and then observe that $\psitil$ remains a cocycle conjugacy for the natural actions of $\Lambda$. We identify $N_i \rtimes_{\hat{\sigma}^i} \Gamma$ with $P_i \ovt B(\ell^2(\Gamma))$ through the Takesaki duality isomorphism given in \eqref{eq:takesakiduality}. Under this identification, the dual weight becomes $\phi_i \ot \omega$ where $\omega$ is the weight on $B(\ell^2(\Gamma))$ given by $\omega = \Tr(M_{\hat{\iota}}\fdot)$. The $\Lambda$-action is now given by $\al^i \ot \id$ and $N_i$ corresponds to the subalgebra $(P_i \ovt B(\ell^2(\Gamma)))_{\phi_i \ot \om}$. From now on, we denote the latter as $N_i$.

Therefore under these identifications, $\psitil$ becomes an isomorphism
$$\Psi : P_0 \ovt B(\ell^2(\Gamma)) \recht P_1 \ovt B(\ell^2(\Gamma))$$
with the following properties: $\Psi$ is a cocycle conjugacy for the actions $\al^i \ot \id$, we have  that $(\phi_1 \ot \om) \circ \Psi = \kappa \, (\phi_0 \ot \om)$, and $\Psi(N_0) = N_1$.

Let $e \in B(\ell^2(\Gamma))$ be the projection onto $\mathbb{C}\delta_1 \subset \ell^2(\Gamma)$. Then $1 \ot e$ is a projection in $N_0$ with weight $1$. Therefore, $\Psi(1 \ot e)$ is a projection in $N_1$ with weight $\kappa \leq 1$. Since $N_1$ is a $\II_\infty$ factor, we can replace $\Psi$ by $(\Ad v) \circ \Psi$ for some unitary $v \in \cU(N_1)$ and assume that $\Psi(1 \ot e) = p \ot e$ where $p$ is a projection in $(P_1)_{\phi_1}$ with $\phi_1(p) = \kappa$. Restricting $\Psi$ to
$$P_0 = (1 \ot e)(P_0 \ovt B(\ell^2(\Gamma)))(1 \ot e)$$
gives the conclusion of the lemma.
\end{proof}

\subsection{Noncommutative Bernoulli actions and their properties}

Let $(\P,\phi)$ be a von Neumann algebra with a normal faithful state $\phi$. Whenever $I$ is a countable set, we write $\P^I$ for the tensor product of $\P$ indexed by $I$ with respect to $\phi$. The canonical product state on $\P^I$ will be denoted by $\phi^I$.
Consider now a countable group $\Lambda$ that acts on $I$, and let $\Lambda$ act on $\P^I$ by the (generalized) Bernoulli action
\begin{align*}
\rho(s)\big(\otimes_{k \in I} a_k\big) &= \otimes_{k \in I} a_{s^{-1} \cdot k},&\text{for }s \in \Lambda, a_h \in \P.
\end{align*}
The von Neumann algebra $(\P,\phi)$ is called the \emph{base algebra} for the Bernoulli action, and the crossed product $\P^I \rtimes \Lambda$ is called the \emph{Bernoulli crossed product}.

The following lemma is a generalization of \cite[Theorem 4.4]{connes-stormer;entropy-automorphisms}
and \cite[Proposition 2.2.2]{popa;rigidity-non-commutative-bernoulli}. For the convenience of the reader, we provide a complete proof.

\begin{lemma}\label{lemma:centralizerfactor}
Let $(\P,\phi)$ be a factor, equipped with a normal faithful almost periodic state $\phi$.
Let $I$ be a countably infinite set. Then $((\P^I)_{\phi^I})' \cap \P^I = \mathbb{C}1$, hence the von Neumann algebra $\P^I$ has a factorial discrete decomposition.
\end{lemma}
\begin{proof}
We can take $I = \mathbb{N}$. Denote $(M,\vphi) = (P,\phi)^\mathbb{N}$. Let $\pi_n : \P \to \M$ be the canonical embedding of $\P$ at the $n$-th position. Let $\Gamma \subset \mathbb{R}^+_0$ be the subgroup generated by the point spectrum of  $\Delta_{\vphi}$, and denote by $G= \hat{\Gamma}$ its compact dual group. We then have the modular automorphism groups $\sigma^\P : G \curvearrowright (\P,\phi)$, $\sigma^\M : G \curvearrowright (\M,\vphi)$, and $\sigma^{\P \otimes \M} = \sigma^\P \otimes \sigma^\M : G \curvearrowright (\P \otimes \M, \phi \otimes \vphi)$.

For any $n \in \mathbb{N}$, consider the state-preserving $\star$-isomorphism $\alpha_n : \M \to \P \otimes \M$
defined by
\begin{align*}
  \alpha_n( \otimes_k x_k) = x_n \otimes \big(x_0 \otimes x_1 \otimes \cdots \otimes x_{n-1} \otimes x_{n+1} \otimes \cdots \big).
\end{align*}
Remark that for every $x \in \M$, $\alpha_n(x) \to 1\otimes x$ $\star$-strongly.
Let now $x \in \M \cap \M_\vphi'$. Since $\alpha_n(x) \in \P \otimes \M \cap (\P \otimes \M)_{\phi \otimes \vphi}'$ for all $n \in \mathbb{N}$, we have $1 \otimes x \in  \P \otimes \M \cap (\P \otimes \M)_{\phi \otimes \vphi}'$.
Take $y \in \bigcup_{\gamma \in \Gamma} \P_{\phi,\gamma}$. We have that $y^\star \otimes \pi_n(y) \in (\P\otimes\M)_{\phi\otimes\vphi}$. This means that $x$ commutes with $\pi_n(y)$.
Since the linear span of $\bigcup_{\gamma \in \gamma} \P_{\phi,\gamma}$ is $\star$-strongly dense in $\P$, we conclude that $x \in \M \cap \M' = \mathbb{C}1$.
\end{proof}

Also the following lemma is probably well known. It provides a criterion for the generalized Bernoulli action to be properly outer.

\begin{lemma}\label{lemma:bernoulli-outer}
Let $(P,\phi)$ be a nontrivial von Neumann algebra equipped with a normal faithful state $\phi$. Let $I$ be a countable set and $\al : I \recht I$ a permutation that moves infinitely many elements of $I$, i.e.~the set of $i \in I$ with $\al(i)\not=i$ is infinite. Then the induced automorphism $\rho \in \Aut(P^I,\phi^I)$ is properly outer: if $v \in P^I$ and $v x = \rho(x) v$ for all $x \in P^I$, then $v = 0$.

More generally, we have the following: if $\theta \in \Aut(P,\phi)$ is a state preserving automorphism with diagonal product $\theta^I \in \Aut(P^I,\phi^I)$, and if $v \in P^I$ such that $v \theta^I(x) = \rho(x) v$ for all $x \in P^I$, then $v=0$.
\end{lemma}
\begin{proof}
We denote $\|x\|_\phi = \sqrt{\phi(x^\star x)}$ for all $x \in P$, and similarly for $x \in P^I$ using $\varphi = \phi^I$. Fix a nonzero element $a \in P$ with $\phi(a) = 0$ and such that both the left and the right multiplication by $a$ have norm less than $1$ on $L^2(P,\phi)$. For every $i \in I$, denote by $\pi_i : P \recht P^I$ the embedding in position $i$. Write $I$ as an increasing union of finite subsets $I_n \subset I$ and denote by $E_n$ the unique state preserving conditional expectation of $P^I$ onto $P^{I_n}$.

Assume that $v \in P^I$ is a nonzero element satisfying $v \theta^I(x) = \rho(x) v$ for all $x \in P^I$. Write $\delta = \frac{1}{2} \, \|v\|_\vphi \, \|a\|_\phi$ and note that $\delta > 0$. Take $n$ large enough such that $\|v - E_n(v)\|_\vphi < \delta$ and $\sqrt{2} \|E_n(v)\|_\vphi \geq \|v\|_\vphi$. Since $\al$ moves infinitely many elements of $I$, we can fix an $i \in I$ such that both $i$ and $\al(i)$ belong to $I \setminus I_n$ and $i \neq \al(i)$. Write $v_n = E_n(v)$.

Note that $\|v \pi_i(\theta(a)) - v_n \pi_i(\theta(a))\|_\vphi \leq \|v - v_n\|_\vphi < \delta$. We also have that $\|\pi_{\al(i)}(a) v - \pi_{\al(i)}(a) v_n\|_\vphi \leq \|v-v_n\|_\vphi < \delta$.
Observe that $\pi_i(\theta(a)) = \theta^I(\pi_i(a))$ and $\pi_{\al(i)}(a) = \rho(\pi_i(a))$. Since $v \theta^I(\pi_i(a)) = \rho(\pi_i(a)) v$, it follows that $\|v_n \pi_i(\theta(a)) - \pi_{\al(i)}(a) v_n\|_\vphi < 2\delta$. Since $\{i\}$, $\{\al(i)\}$ and $I_n$ are disjoint subsets of $I$ and since $\phi(a) = 0$, we get that
$$2\delta > \|v_n \pi_i(\theta(a)) - \pi_{\al(i)}(a) v_n\|_\vphi = \sqrt{2} \, \|v_n\|_\vphi \, \|a\|_\phi \geq \|v\|_\vphi \, \|a\|_\phi \; .$$
We reached the absurd conclusion that $2\delta > \|v\|_\vphi \, \|a\|_\phi$.
\end{proof}

Finally, the following lemmas provide a criterion for a generalized Bernoulli crossed product to be full.
We first recall the following terminology. An \emph{invariant mean} for an action of a countable group $\Lambda$ on a set $X$ is a finitely additive $\Lambda$-invariant probability measure on all the subsets of $X$. If $\pi : \Lambda \to \mathcal{U}(H)$ is a unitary representation, a sequence of unit vectors $\xi_n \in H$ is called \emph{almost invariant} if $\|\xi_n - \pi(g)(\xi_n)\|\to 0$ for all $g \in \Lambda$. Note that $\Lambda \curvearrowright X$ admits an invariant mean if and only if the unitary representation $\Lambda \curvearrowright \ell^2(X)$ admits almost invariant unit vectors. A unitary representation $\pi : \Lambda \to \mathcal{U}(H)$ is called \emph{amenable} if there exists a state $\omega$ on $B(H)$ such that $\omega(T) = \omega(\pi(g)T\pi(g^{-1}))$ for all $g \in \Lambda, T \in B(H)$, see \cite{bekka;amenable-unitary-representations}.

\begin{lemma}\label{lem.mean}
Let $(P,\phi)$ be a nontrivial von Neumann algebra equipped with a normal faithful state $\phi$. Let $\Lambda \actson I$ be an action of the countable group $\Lambda$ on the countable set $I$. The following statements are equivalent.
\begin{enumerate}[\upshape(i)]
\item The action $\Lambda \actson I$ admits an invariant mean.
\item The action of $\Lambda$ on the set of all nonempty finite subsets of $I$ admits an invariant mean.
\item The unitary representation $\Lambda \actson L^2((P,\phi)^I \ominus \C 1)$ admits almost invariant unit vectors.
\item The unitary representation $\Lambda \actson L^2((P,\phi)^I \ominus \C 1)$ is amenable.
\end{enumerate}
\end{lemma}
\begin{proof}
We write $H = L^2((P,\phi)^I \ominus \C 1)$ and denote by $\rho : \Lambda \actson H$ the unitary representation given by the Bernoulli action. We denote by $J$ the set of all nonempty finite subsets of $I$.

(i) $\Rightarrow$ (iii). Fix $a \in P \ominus \C1$ with $\phi(a^\star a) = 1$. Let $\xi_n \in \ell^2(I)$ be a sequence of finitely supported unit vectors satisfying $\lim_n \|g \cdot \xi_n - \xi_n\|_2 = 0$ for all $g \in \Lambda$. For every $i \in I$, denote by $\pi_i : P \recht P^I$ the embedding as the $i$-th tensor factor. Define $\eta_n \in P^I$ given by
$\eta_n = \sum_{i \in I} \xi_n(i) \pi_i(a)$. Then $\eta_n$ is a sequence of almost invariant unit vectors in $H$.

(iii) $\Rightarrow$ (iv) is trivially true.

(iv) $\Rightarrow$ (ii). For every finite subset $\cF \subset I$, denote by $p_\cF$ the orthogonal projection of $H$ onto the closed linear span of $(P \ominus \C 1)^\cF$. Define the map $\Theta : \ell^\infty(J) \recht B(H)$ given by $\Theta(F) = \sum_{\cF \in J} F(\cF) p_\cF$. Since $\Theta(g \cdot F) = \rho(g) \Theta(F) \rho(g)^\star$, the composition of an $\Ad \rho(\Lambda)$-invariant mean on $B(H)$ with $\Theta$ gives a $\Lambda$-invariant mean on $J$.

(ii) $\Rightarrow$ (i). For every $k \geq 1$, define $V_k \subset I^k$ as the subset of $k$-tuples that consist of $k$ distinct elements of $I$. Put $V = \sqcup_{k \geq 1} V_k$. The formula $\theta(i_1,\ldots,i_k) = \{i_1,\ldots,i_k\}$ defines a finite-to-one $\Lambda$-equivariant map $V \recht J$. So also the action $\Lambda \actson V$ admits an invariant mean $m$. We push forward $m$ along the $\Lambda$-equivariant map $V \recht I$ given by $(i_1,\ldots,i_k) \mapsto i_1$ and find a $\Lambda$-invariant mean on $I$.
\end{proof}

Connes showed in \cite[Proposition 3.9 (b)]{connes;almost-periodic-states} that any Bernoulli crossed product $(\P,\phi)^{\F_2} \rtimes \F_2$ of $\F_2$ is a full factor.
In fact, any Bernoulli crossed product of a nonamenable group is a full factor, as we see now. For completeness, we provide a complete proof.
Recall that a bounded sequence $(x_n)_{n\in \mathbb{N}}$ in a von Neumann algebra $\M$ is called \emph{central} if for all $y \in \M$, $x_ny-yx_n \to 0$ $\star$-strongly.

\begin{lemma}\label{lemma.full-bernoulli}\label{lemma:fullfactor}
Let $(\P,\phi)$ be a nontrivial von Neumann algebra equipped with a normal faithful state. Let $\Lambda$ be a countable group acting on the countable set $I$, such that $\Lambda \curvearrowright I$ has no invariant mean, and consider the Bernoulli action $\Lambda \curvearrowright^\rho \P^I$. Denote by $\tau$ the trace on $L\Lambda$.
The following two statements are equivalent.
\begin{enumerate}[\upshape(i)]
\item Every $\|\cdot\|_\infty$-bounded central sequence $(y_n)_{n \in \mathbb{N}} \in L\Lambda$ satisfying $\|y_n - E_{L(\Stab i)}(y_n)\|_2 \to 0$ for every $i \in I$, must satisfy $\|y_n - \tau(y_n)1\|_2 \recht 0$.
\item $\P^I \rtimes \Lambda$ is a full factor.
\end{enumerate}
If these conditions hold, then $\tau(\P^I \rtimes \Lambda)$ is the weakest topology on $\mathbb{R}$ that makes the function $\sigma^{\phi} : \mathbb{R} \to \Aut(\P) : t \mapsto \sigma^{\phi}_t$ continuous.
\end{lemma}

Remark that condition (i) is trivially fulfilled if $I = \Lambda$ and $\Lambda$ acts on the left, since then $\Stab(\{e\})=\{e\}$.

\begin{proof}
Denote by $\varphi = \phi^I \circ E_{\P^I}$ the canonical state on $\P^I \rtimes \Lambda$. We write $\|x\|_\vphi = \vphi(x^\star x)^{\frac{1}{2}}$ and $\|x\|_\vphi^\sharp = \vphi(x^\star x + x x^\star)^{\frac{1}{2}}$. Assume that (i) holds. We start by proving the following statement:
\begin{equation}\label{eq.assum}
\begin{split}
& \text{if $x_n \in P^I \rtimes \Lambda$ is a bounded sequence and $t_n \in \RR$ is any sequence such that} \\
& \|x_n \sigma_{t_n}^\vphi(a) - a x_n\|_\vphi^\sharp \recht 0 \quad\text{for all}\;\; a \in P^I \rtimes \Lambda \; , \quad\text{then $\|x_n - \vphi(x_n)1\|_\vphi^\sharp \recht 0$.}
\end{split}
\end{equation}

Let $(x_n)$ and $(t_n)$ be sequences as in \eqref{eq.assum}. We denote by $(u_g)_{g \in \Lambda}$ the canonical unitaries in $P^I \rtimes \Lambda$. We write $H = L^2(P^I \ominus \C 1,\vphi)$ and denote by $\rho$ the unitary representation $\Lambda \actson H$ given by the generalized Bernoulli action. Identifying $L^2(P^I \rtimes \Lambda)$ with the direct sum of $H \ot \ell^2(\Lambda)$ and $\ell^2(\Lambda)$, the unitary representation $(\Ad u_g)_{g \in \Lambda}$ of $\Lambda$ on $L^2(P^I \rtimes \Lambda)$ becomes the direct sum of $(\rho_g \ot \Ad u_g)_{g \in \Lambda}$ and $(\Ad u_g)_{g \in \Lambda}$. Taking $a = u_g$ in \eqref{eq.assum}, we can view $(x_n)$ as a sequence of almost $(\Ad u_g)_{g \in \Lambda}$-invariant vectors in $L^2(P^I \rtimes \Lambda)$. By \cref{lem.mean}, the representation $\rho$ is nonamenable, so that the representation $(\rho_g \ot \Ad u_g)_{g \in \Lambda}$ does not admit almost invariant unit vectors. We conclude that $\|x_n - E_{L\Lambda}(x_n)\|_\vphi \recht 0$. We write $y_n = E_{L\Lambda}(x_n)$. Applying the previous argument to $x_n^\star$, we conclude that $\|x_n - y_n\|_\vphi^\sharp \recht 0$. Also, $y_n$ is a central sequence in $L\Lambda$.

Fix a nonzero element $b \in P \ominus \C1$ such that the right multiplication with $b$ is a bounded operator on $L^2(P,\phi)$. Fix $i \in I$ and denote by $\pi_i : P \recht P^I$ the embedding as the $i$-th tensor factor. Taking $a = \pi_i(b)$ in \eqref{eq.assum}, we get that
$$\|y_n \, \pi_i(\sigma_{t_n}^\phi(b)) - \pi_i(b) \, y_n\|_\vphi^\sharp \recht 0 \; .$$
A direct computation gives that
$$\sqrt{2} \; \|b\|_\phi^\sharp \; \|y_n - E_{L(\Stab i)}(y_n)\|_2 \leq \|y_n \, \pi_i(\sigma_{t_n}^\phi(b)) - \pi_i(b) \, y_n\|_\vphi^\sharp \; .$$
It thus follows that $\|y_n - E_{L(\Stab i)}(y_n)\|_2 \recht 0$ for every $i \in I$. Because (i) holds, we get that $\|y_n - \tau(y_n)1\|_2 \recht 0$. This concludes the proof of \eqref{eq.assum}.

Taking $t_n = 0$ in \eqref{eq.assum}, it follows in particular that $P^I \rtimes \Lambda$ is a full factor. Denote by $\tau$ the weakest topology on $\RR$ that makes the function $\sigma^{\phi} : \mathbb{R} \to \Aut(\P) : t \mapsto \sigma^{\phi}_t$ continuous. We prove that $\tau = \tau(P^I \rtimes \Lambda)$. Because $\Aut(P)$ and $\Out(P^I \rtimes \Lambda)$ are Polish groups, it suffices to prove that if $t_n \recht 0$ in $\tau(P^I \rtimes \Lambda)$, then also $t_n \recht 0$ in $\tau$. In that case, there is a sequence of unitaries $u_n \in P^I \rtimes \Lambda$ such that $(\Ad u_n) \circ \sigma_{t_n}^\vphi \recht \id$ in $\Aut(P^I \rtimes \Lambda)$. It follows from \eqref{eq.assum} that $\|u_n - \vphi(u_n)1\|_\vphi^\sharp \recht 0$. This means that $\Ad u_n \recht \id$ in $\Aut(P^I \rtimes \Lambda)$. Then also $\sigma_{t_n}^\vphi \recht \id$ in $\Aut(P^I \rtimes \Lambda)$. Restricting $\sigma_{t_n}^\vphi$ to one copy of $P$, it follows that $t_n \recht 0$ in $\tau$.

Conversely, assume that $P^I \rtimes \Lambda$ is full. We prove that (i) holds. Let $y_n \in L\Lambda$ be a bounded central sequence satisfying $\|y_n - E_{L(\Stab i)}(y_n)\|_2 \to 0$ for every $i \in I$. Denote by $P_0 \subset P$ the set of elements $b \in P$ such that the right multiplication with $b$ is a bounded operator on $L^2(P,\phi)$. Define $M_0$ as the linear span of all $P_0^\cF u_g$, $\cF \subset I$ finite and $g \in \Lambda$. It follows that $\|y_n a - a y_n\|_\vphi^\sharp \recht 0$ for all $a \in M_0$. Since moreover $y_n \vphi = \vphi y_n$ for all $n$, it follows from \cite[2.8 and 3.1]{connes;almost-periodic-states} that $\|y_n - \tau(y_n)1\|_2 \recht 0$.
\end{proof}

\section{Isomorphism results for type \texorpdfstring{$\III$}{III} Bernoulli crossed products}\label{section:isomorphism}

In \cite{ocneanu;actions-amenable-groups}, Ocneanu obtained a complete classification up to cocycle conjugacy of actions of amenable groups on the hyperfinite $\II_1$ or $\II_\infty$ factor. We deduce as a corollary of Ocneanu's work a classification result for actions of amenable groups on amenable factors equipped with an almost periodic state. We work with state preserving actions and look for state preserving cocycle conjugacies.

It is essential for us to work in a state preserving setting throughout. The reason is that in a second step, we develop a noncommutative version of Bowen's co-induction method \cite{bowen;orbit-equivalence-coinduced-actions-free-products} to pass from cocycle conjugacy of noncommutative Bernoulli $\Sigma$-actions to cocycle conjugacy of Bernoulli $\Sigma \ast \Upsilon$-actions. This method only works if the original cocycle conjugacy is state preserving. At the end of this section, we then find the `isomorphism part' of \cref{thm.A}.

\begin{theorem}\label{theorem:cocycleconjugateZ}\label{theorem:cocycleconjugateamenablecorner}
Let $(\M_0,\phi_0)$, $(\M_1,\phi_1)$ be nontrivial amenable factors, with normal faithful almost periodic states having factorial discrete decompositions. Let $\Sigma$ be a countably infinite amenable group, with state preserving actions $\Sigma \curvearrowright^{\alpha^i} (\M_i,\phi_i)$, for $i=0,1$. Assume that the restrictions of $\alpha^i$ to $(M_i)_{\phi_i}$ are outer.

The actions $\alpha^0$ and $\alpha^1$ are cocycle conjugate through a state preserving isomorphism if and only if the point spectra of $\Delta_{\phi_i}$ coincide.
\end{theorem}
\begin{proof}
One implication being obvious, assume that the point spectra of $\Delta_{\phi_i}$ coincide. Denote by $G$ the canonically associated compact group with the modular automorphism groups $(\si^i_s)_{s \in G}$. When $G=\{1\}$, the $\phi_i$ are traces and the theorem is exactly \cite[Corollary 1.4]{ocneanu;actions-amenable-groups}. So we may assume that the $\phi_i$ are not traces. Since the actions $\al^i$ are state preserving, they canonically extend to actions $\altil^i$ of $\Sigma$ on $N_i = M_i \rtimes_{\si^i} G$. Note that the action $\altil^i$ commutes with the dual action $(\hat{\si}_\gamma)_{\gamma \in \Gamma}$ of $\Gamma = \hat{G}$. Both combine into an action $(\beta^i_{(\sigma,\gamma)})_{(\sigma,\gamma) \in \Sigma \times \Gamma}$ on the hyperfinite $\II_\infty$ factor $N_i$.

We want to apply \cite[Theorem 2.9]{ocneanu;actions-amenable-groups} to get that $\beta^0$ and $\beta^1$ are cocycle conjugate. For this, it suffices to check that the $\beta^i$ are outer actions that scale the trace in the same way. The latter follows because $\mod \beta^i_{(\sigma,\gamma)} = \hat{\iota}(\gamma)^{-1}$. To prove the former, assume that $\beta^i_{(\sigma,\gamma)}$ is inner. Since this automorphism scales the trace with the factor $\hat{\iota}(\gamma)^{-1}$, we get that $\gamma = 1$. So, $\altil^i_\sigma$ is inner. It then follows from \cref{lemma:outer-on-core} that the restriction of $\al^i_\sigma$ to $(M_i)_{\phi_i}$ is inner. Therefore, $\sigma = e$.

We claim that the actions $\beta^i$ are cocycle conjugate through a trace preserving isomorphism. To prove this claim, denote by $R_\infty$ the hyperfinite $\II_\infty$ factor. Also the action $\beta^1 \ot \id$ on $N_1 \ovt R_\infty$ is outer and scales the trace in the same way as the actions $\beta^i$. Applying twice \cite[Theorem 2.9]{ocneanu;actions-amenable-groups}, we find a cocycle conjugacy $\psi_0 : N_0 \recht N_1 \ovt R_\infty$ between $\beta^0$ and $\beta^1 \ot \id$, as well as a cocycle conjugacy $\psi_1 : N_1 \ovt R_\infty \recht N_1$ between $\beta^1 \ot \id$ and $\beta^1$. Choosing an automorphism $\theta \in \Aut(R_\infty)$ with $\mod \theta = \mod \psi_0 \cdot \mod \psi_1$, it follows that $\psi = \psi_1 \circ (\id \ot \theta^{-1}) \circ \psi_0$ is a trace preserving cocycle conjugacy between $\beta^0$ and $\beta^1$. The conclusion follows from \cref{lemma:cocycleconjugatecore}.
\end{proof}

We now recall the construction of co-induced actions, which we use to lift the result of \cref{theorem:cocycleconjugateZ} to Bernoulli actions of free product groups.
Let $\Sigma \curvearrowright^\alpha (\M,\phi)$ be a countable group acting on a von Neumann algebra by state preserving automorphisms. Assume that $\Sigma$ is a subgroup of a countable discrete group $\Lambda$.
Choose a map $r : \Lambda \to \Sigma$ such that $r(g \sigma) = r(g) \sigma$ for all $g \in \Lambda, \sigma \in \Sigma$ and such that $r(e)=e$. We have the associated $1$-cocycle $\Omega : \Lambda \times \Lambda/\Sigma \to \Sigma$ for the left action of $\Lambda$ on $\Lambda/\Sigma$, given by $\Omega(g, h\Sigma) =r(gh)r(h)^{-1}$ for all $g,h \in \Lambda$.
Let $\M^{\Lambda / \Sigma}$ be the tensor product of $\M$ indexed by $\Lambda/\Sigma$, with respect to the state $\phi$, and denote by $\pi_{h\Sigma} : \M \to \M^{\Lambda / \Sigma}$ the embedding of $\M$ at position $h\Sigma$, for $h \in \Lambda$.
The formula
\begin{align*}
\Lambda \curvearrowright^\beta \M^{\Lambda/\Sigma} \text{ where } \beta_g\big(\pi_{h \Sigma}(a)\big) = \pi_{gh\Sigma}\big( \alpha_{\Omega(g,h\Sigma)}(a)\big), \quad a \in \M,\: g,h \in \Lambda,
\end{align*}
yields a well defined state preserving action of $\Lambda$ on the tensor product $ (\M^{\Lambda/\Sigma }, \phi^{\Lambda/\Sigma})$, called the \emph{co-induced action} of $\Sigma \curvearrowright (\M,\phi)$ to $\Lambda$.

The following proposition provides a variant of Bowen's co-induction argument in \cite{bowen;orbit-equivalence-coinduced-actions-free-products}, with a proof in the spirit of \cite{meesschaert-raum-vaes}.

\begin{proposition}
Let $\Sigma$ and $\Upsilon$ be countable groups, and $(\M_0,\phi_0), (\M_1,\phi_1)$ be von Neumann algebras equipped with normal faithful states. Assume that $\Sigma \curvearrowright^{\alpha^i} (\M_i,\phi_i)$ are state preserving actions, for $i=0,1$. Denote $\Lambda = \Sigma \ast \Upsilon$, and let $\Lambda \curvearrowright^{\beta^i} \M_i^{\Lambda / \Sigma}$ be the co-induced actions.

If the actions $\alpha^i$ are cocycle conjugate through a state preserving isomorphism, then also the co-induced actions $\beta^i$ are cocycle conjugate through a state preserving isomorphism.
\label{proposition:coinduction}
\end{proposition}
\begin{proof}
Consider two cocycle conjugate actions $\alpha^i$ of $\Sigma$ and their co-inductions $\beta^i$ to $\Lambda = \Sigma \ast \Upsilon$, as in the statement of the theorem.
Let $\psi : \M_0 \to \M_1$ be a state preserving isomorphism and $(w_\sigma)_{\sigma \in \Sigma}$ be a 1-cocycle for $\alpha^1$ such that
\begin{align*}
\psi \circ \alpha^0_\sigma &= \Ad w_\sigma \circ \alpha^1_\sigma \circ \psi,\qquad\forall \sigma \in \Sigma.
\end{align*}
Consider the action $\tilde\alpha^0$ of $\Sigma$ on $\M_1$ given by
$\tilde\alpha^0_\sigma = \psi \circ \alpha^0_\sigma \circ \psi^{-1}$, for $\sigma \in \Sigma$.
The infinite tensor product $\psi^{\Lambda/\Sigma}$ is a state preserving conjugacy between the co-induction of $\alpha^0$ and the co-induction of $\tilde\alpha^0$. So, it is enough to show that the co-induced actions of $\tilde\alpha^0$ and $\alpha^1$ to $\Lambda$ are cocycle conjugate through a state preserving isomorphism. This allows us to simplify the notations: we put $\M = \M_1$, $\alpha = \alpha^1$ and $\tilde\alpha = \tilde\alpha^0$. We then have
\begin{align*}
\tilde\alpha_\sigma = \Ad w_{\sigma} \circ \alpha_\sigma, \qquad \sigma \in \Sigma.
\end{align*}

Take $\mathcal{I} \subset \Lambda = \Sigma \ast \Upsilon$ to be the left transversal of $\Sigma < \Lambda$ consisting of the trivial word $e$ and the words ending in a letter of $\Upsilon -\{e\}$. Consider the map $r : \Lambda \to \Sigma$ defined by $r(g\sigma) = \sigma$ whenever $g \in \mathcal{I}$ and $\sigma \in \Sigma$, and let $\Omega : \Lambda \times \Lambda/\Sigma \to \Sigma$, $\Omega(g,h\Sigma) = r(gh)r(h)^{-1}$ be the associated $1$-cocycle. Denote by $\beta$ and $\tilde\beta$ the co-induced actions to $\Lambda$ of $\alpha$ and $\tilde\alpha$ respectively, with respect to $\Omega$.
Remark that for $\sigma \in \Sigma$, $\upsilon \in \Upsilon$, $h \in \mathcal{I} - \{e\}$,
\begin{align*}
\Omega(\sigma, h\Sigma) &= e \qquad \text{because $\sigma (\mathcal{I} - \{e\}) \subset \mathcal{I} - \{e\}$},\\
\Omega(\sigma, e\Sigma) &= \sigma, &\\
\Omega(\upsilon, h\Sigma) &= e \qquad \text{because $\upsilon \mathcal{I} \subset \mathcal{I}$},\\
\Omega(\upsilon, e\Sigma) &= e.
\end{align*}

Since $\Lambda = \Sigma \ast \Upsilon$ is the free product of $\Upsilon$ and $\Sigma$, there is a unique 1-cocycle $W_g \in \M^{\Lambda/\Sigma}$, $g \in \Lambda$ for $\beta$ given by
\begin{align*}
W_\sigma &= \pi_{e\Sigma}(w_\sigma), \quad \sigma \in \Sigma,&
W_\upsilon &= 1, \quad \upsilon \in \Upsilon.
\end{align*}
We claim that $\tilde\beta_g = \Ad W_{g} \circ \beta_g$ for all $g \in \Lambda$, from which the lemma follows. It suffices to prove that this equation holds for all elements of $\Sigma$ and $\Upsilon$.
Let $\sigma \in \Sigma$, then, for $h \in \mathcal{I} - \{e\}$, $a \in \M$,
\begin{align*}
\tilde\beta_\sigma \big(\pi_{h \Sigma}(a)\big) &= \pi_{\sigma h\Sigma}\big(\tilde\alpha_{e}(a)\big) = \pi_{\sigma h \Sigma}(a),\\
\Ad W_\sigma \circ \beta_\sigma \big(\pi_{h \Sigma}(a)\big) &= \pi_{e\Sigma}(w_\sigma)\pi_{\sigma h \Sigma}\big(\alpha_{e}(a)\big)\pi_{e\Sigma}(w_\sigma^\star) = \pi_{\sigma h \Sigma}(a),
\end{align*}
and also
\begin{align*}
\tilde\beta_\sigma\big(\pi_{e\Sigma}(a)\big)
= \pi_{e\Sigma}\big(\tilde\alpha_\sigma(a)\big)
= \pi_{e\Sigma}\big(w_\sigma \alpha_\sigma(a) w_\sigma^\star\big)
= W_\sigma \pi_{e\Sigma}\big(\alpha_{\sigma}(a)\big) W_\sigma^\star
= \Ad W_\sigma \circ \beta_\sigma\big(\pi_{e\Sigma}(a)\big),
\end{align*}
which means that $\tilde\beta_\sigma(x) =\Ad W_\sigma \circ \beta_\sigma(x)$ for all $x \in \M^{\Lambda/\Sigma}$.
Let now $\upsilon \in \Upsilon$, and $h \in \mathcal{I}$, $a \in \M$, then
\begin{align*}
\tilde\beta_\upsilon\big(\pi_{h \Sigma}(a)\big) &= \pi_{\upsilon h \Sigma}\big( \tilde\alpha_e(a)\big)
  = \pi_{\upsilon h \Sigma}\big( a\big)
 = \pi_{\upsilon h \Sigma}\big( \alpha_e(a)\big)
  = \Ad W_\upsilon \circ \beta_\upsilon\big(\pi_{h\Sigma}(a)\big),
\end{align*}
so $\tilde\beta_\upsilon = \Ad W_\upsilon \circ \beta_\upsilon$ as well.
\end{proof}

The previous two results now yield the following theorem, providing the `isomorphism part' of \cref{thm.A}.

\begin{theorem}\label{thm:cocycle-conjugacy-free-products}
Suppose that $(\P_0,\phi_0)$ and $(\P_1,\phi_1)$ are two amenable factors, equipped with normal faithful almost periodic states $\phi_0, \phi_1$. Let $\Sigma$ be a countably infinite amenable group, and $\Upsilon$ any countable group.
The following two statements are equivalent.
\begin{enumerate}[\upshape(i)]
\item The point spectra of $\Delta_{\phi_0}$ and $\Delta_{\phi_1}$ generate the same subgroup of $\mathbb{R}^+_0$.
\item With $\Lambda = \Sigma \ast \Upsilon$, the Bernoulli actions $\Lambda \curvearrowright (P_0,\phi_0)^\Lambda$ and $\Lambda \curvearrowright (P_1,\phi_1)^\Lambda$ are cocycle conjugate through a state preserving isomorphism.
\end{enumerate}
\label{theorem:cocycleconjugateFn}
\end{theorem}
\begin{proof}
If (ii) holds, it follows that $\Delta_{\phi_0^\Lambda}$ and $\Delta_{\phi_1^\Lambda}$ have the same point spectrum, which exactly means that (i) holds.
We prove the implication (i) $\Rightarrow$ (ii).

Denote by $\rho^i$ the Bernoulli action of $\Sigma$ on $\P_i^\Sigma$, for $i=0,1$.
The factors $\P_0^\Sigma$, $\P_1^\Sigma$ are amenable, and they have normal faithful almost periodic states and factorial discrete decompositions, by \cref{lemma:centralizerfactor}. Moreover, the point spectra of $\Delta_{\phi_0^\Sigma}$, $\Delta_{\phi_1^\Sigma}$ form the same subgroup of $\mathbb{R}_0^+$. Combining lemmas \ref{lemma:outer-on-core} and \ref{lemma:bernoulli-outer},
the Bernoulli actions of $\Sigma$ on $\P_i^\Sigma$ are state preserving and outer when restricted to the centralizer of $\phi_i^\Sigma$. By \cref{theorem:cocycleconjugateamenablecorner}, the Bernoulli actions of $\Sigma$ on $\P_0^\Sigma$ and $\P_1^\Sigma$ are cocycle conjugate through a state preserving isomorphism.

Note that whenever $\Sigma < \Lambda$ is a subgroup, the co-induced action of the Bernoulli action $\Sigma \curvearrowright (\P_i, \phi_i)^{\Sigma}$ to $\Lambda$ is exactly the Bernoulli action $\Lambda \curvearrowright (\P_i,\phi_i)^{\Lambda}$. By \cref{proposition:coinduction}, (ii) follows.
\end{proof}

\section{\boldmath Uniqueness of regular amenable subalgebras and a class $\cC$ of groups}\label{section:regular-subalgebras}

We first recall from \cite[Theorem 2.1 and Corollary 2.3]{popa;strong-rigidity-malleable-actions-I} Popa's theory of intertwining-by-bimodules.

Let $(\M,\tau)$ be any tracial von Neumann algebra, and let $\P \subset 1_\P \M 1_\P$ and $\Q \subset 1_\Q \M 1_\Q$ be von Neumann subalgebras. We write $\P \prec_\M \Q$ if there exist a projection $p \in M_n(\C) \ot Q$, a normal unital $\star$-homomorphism $\theta : P \recht p(M_n(\C) \ot Q)p$ and a non-zero partial isometry $v \in M_{1,n}(\C) \ot 1_P M 1_Q$ satisfying $a v = v \theta(a)$ for all $a \in P$.

For every von Neumann subalgebra $P\subset M$, we consider the \emph{normaliser} $\mathcal{N}_M(P) = \{u \in \mathcal{U}(M) \mid uPu^\star = P\}$. When $\mathcal{N}_M(P)\dpr = M$, we say that $P \subset M$ is \emph{regular}. When moreover $P \subset M$ is maximal abelian, we call $P \subset M$ a Cartan subalgebra. Finally, when $P$ and $M$ are $\II_1$ factors, we say that $P \subset M$ is an irreducible subfactor if $P' \cap M = \C 1$.

We will use the notion of intertwining-by-bimodules in combination with the following lemma.

\begin{lemma}[\citep{Lemma 8.4}{ioana-peterson-popa;amalgamated-free-product}] \label{lemma:weakcontainmentisunitaryconjug}
Let $\M$ be a $\II_1$ factor and $\P_0,\P_1 \subset \M$ regular, irreducible subfactors. Assume that the groups $\frac{\mathcal{N}_\M(\P_i)}{\mathcal{U}(\P_i)}$ are icc.
Suppose that
\begin{align*}
\P_0 \prec \P_1 \quad \text{ and } \quad \P_1 \prec \P_0.
\end{align*}
Then there exists a unitary $u \in \mathcal{U}(\M)$ such that $u\P_0u^\star = \P_1$.
\end{lemma}

Using the main results of \cite{popa-vaes;unique-cartan-decomposition-factors-free,popa-vaes;unique-cartan-decomposition-factors-hyperbolic}, the assumptions $P_0 \prec P_1$ and $P_1 \prec P_0$ in the previous lemma are automatically satisfied when the algebras $P_i$ are amenable and the groups $\frac{\mathcal{N}_\M(\P_i)}{\mathcal{U}(\P_i)}$ are hyperbolic (or satisfy other rank $1$ type conditions). The same holds for many other classes of groups and, as we will see, they all belong to one natural family of groups that we introduce now.

Let $(M,\tau)$ be a tracial von Neumann algebra. The \emph{Jones index} of a von Neumann subalgebra $P \subset M$ is defined as the $P$-dimension of $L^2(M)$ as a $P$-module, computed using the given trace $\tau$.
An inclusion $P \subset M$ is said to be \emph{essentially of finite index} if there exist projections $p \in P' \cap M$ that lie arbitrarily close to $1$ such that $P p \subset pMp$ has finite Jones index.
Note that being essentially of finite index is independent of the choice of trace $\tau$. For later use, also observe that for abelian von Neumann algebras $A \subset B$, being essentially of finite index is equivalent to having a family of projections $p_k \in B$ such that $\sum_k p_k = 1$ and $Ap_k=Bp_k$ for all $k$.

We call $A \subset M$ a \emph{virtual core subalgebra} if $A' \cap M  = \cZ(A)$ and if the inclusion $\cN_M(A)\dpr \subset M$ is essentially of finite index.

\begin{definition}\label{definition:class-C}
We say that a countably infinite group $\Gamma$ belongs to class $\cC$ if for every trace preserving cocycle action $\Gamma \actson (B,\tau)$ and every amenable, virtual core subalgebra $A \subset p (B \rtimes \Gamma)p$, we have that $A \prec B$.
\end{definition}

Note that all groups in the class $\cC$ are nonamenable: this follows by taking $B = \C1$ and $A = L(\Gamma)$. Similarly, groups in the class $\cC$ do not admit infinite amenable normal subgroups.

\begin{remark}\label{remark:many-in-class-C}
We have several motivations to introduce this class $\cC$ of groups. Our definition is first of all motivated by the concept of $\cC_s$-rigidity introduced in \cite[Definition 1.4]{popa-vaes;unique-cartan-decomposition-factors-free}. A group $\Gamma$ is called $\cC_s$-rigid if the following property holds. For every free ergodic probability measure preserving action of $\Gamma$ on $(X,\mu)$ with crossed product $M = L^\infty(X) \rtimes \Gamma$ and for any maximal abelian von Neumann subalgebra $A \subset M$ whose normalizer is a finite index subfactor of $M$, we have that $A$ is unitarily conjugate with $L^\infty(X)$. By \cite[Theorem A.1]{popa;class-factors-betti-numbers}, every group in the class $\cC$ is also $\cC_s$-rigid. In particular, all groups in the class $\cC$ are Cartan-rigid (or $\cC$-rigid), meaning that for all crossed product $\II_1$ factors $L^\infty(X) \rtimes \Gamma$ by free ergodic probability measure preserving actions, $L^\infty(X)$ is the unique Cartan subalgebra up to unitary conjugacy.

Also, by \cref{lemma:weakcontainmentisunitaryconjug}, if $\Gamma$ and $\Lambda$ are icc groups in the class $\cC$ and if we have a stable isomorphism of crossed products $\pi : R \rtimes \Gamma \recht (R \rtimes \Lambda)^t$ given by outer actions of $\Gamma,\Lambda$ on the hyperfinite $\II_1$ factor $R$, then $\pi(R)$ is unitarily conjugate to $R^t$ and $\Gamma \cong \Lambda$. This is the result that we need in our paper.

So, class $\cC$ unifies Cartan-rigidity and unique crossed product decompositions for outer actions. Moreover, as we prove below, the class $\cC$ satisfies several stability properties, making it a natural class of groups to consider. In particular, we prove in \cref{proposition:class-C-stable-ext} that class $\cC$ is stable under extensions. This stability result is motivated by \cite{chifan-ioana-kida;arbitrary-actions-central-quotients-braid-groups}, where the structural theorems of \cite{popa-vaes;unique-cartan-decomposition-factors-free,popa-vaes;unique-cartan-decomposition-factors-hyperbolic} for crossed products with free groups or hyperbolic groups are used to show that also extensions of hyperbolic groups by hyperbolic groups are Cartan-rigid.

Finally note that over the last few years, large classes of groups have been shown to belong to class $\cC$. Assume that $\Gamma \actson (B,\tau)$ is a cocycle action. Write $M = B \rtimes \Gamma$. To `remove' the cocycle, consider the dual coaction $\Delta : M \recht M \ovt L(\Gamma)$ defined by $\Delta(b u_g) = b u_g \ot u_g$ for all $b \in B$, $g \in \Gamma$. Whenever $A \subset p M p$ is an amenable, virtual core subalgebra, we write $q = \Delta(p)$ and we can apply the following results to the inclusion $\Delta(A) \subset q (M \ovt L(\Gamma)) q$. Note that, in order to prove that $A \prec B$, it suffices to show that $\Delta(A) \prec M \ot 1$.
\begin{itemize}
\item By \cite[Theorem 3.1, Lemma 4.1 and Theorem 7.1]{popa-vaes;unique-cartan-decomposition-factors-free}, every weakly amenable group $\Gamma$ with $\beta_1^{(2)}(\Gamma) > 0$ belongs to the class $\cC$.
\item By \cite[Theorem 3.1]{popa-vaes;unique-cartan-decomposition-factors-hyperbolic}, all weakly amenable, nonamenable, bi-exact groups (and in particular, all nonelementary hyperbolic groups) belong to the class $\cC$.
\item By \cite[Theorem 1.6]{ioana;cartan-subalgebras-amalgamated-free-product-factors} (see also \cite[Theorem A]{vaes;normalizers-inside-amalgamated-free-product}), every free product $\Gamma = \Gamma_1 \star \Gamma_2$ with $|\Gamma_1| \geq 2$ and $|\Gamma_2| \geq 3$ belongs to the class $\cC$.
\end{itemize}
\end{remark}

The proof of the following lemma is essentially contained in the proof of \cite[Theorem A.2]{popa;class-factors-betti-numbers}.

\begin{lemma}\label{lemma:lem1}
Let $(M,\tau)$ be a tracial von Neumann algebra and $Q \subset N \subset M$ von Neumann subalgebras. Assume that $q \in Q' \cap M$ is a nonzero projection such that $q(Q' \cap M)q = \cZ(Q) q$. Then there exists a nonzero projection $p \in Q' \cap N$ such that $p(Q' \cap N)p = \cZ(Q) p$ and such that $p$ is smaller than the support projection of $E_N(q)$.
\end{lemma}
\begin{proof}
Denote by $e$ the support projection of $E_N(q)$. Note that $q \leq e$ and that $e \in Q' \cap N$. Replacing $Q \subset N \subset M$ by $Q e \subset eNe \subset eMe$, we may assume that the support projection of $E_N(q)$ equals $1$.

Denote by $z \in \cZ(Q' \cap M)$ the central support of $q \in Q' \cap M$. Since $q (Q' \cap M) q = \cZ(Q) q$, the projection $q \in Q' \cap M$ is abelian and $(Q' \cap M)z$ is of type I with center equal to $\cZ(Q)z$. In particular, the inclusion $\cZ(Q) z \subset (Q' \cap M) z$ is essentially of finite index. A fortiori, the inclusion $\cZ(Q) z \subset (Q' \cap N) z$ is essentially of finite index. This inclusion is isomorphic with the inclusion $\cZ(Q) f \subset (Q' \cap N) f$, where $f$ denotes the support projection of $E_N(z)$. Since $\cZ(Q) f$ lies in the center of $(Q' \cap N)f$, it follows that $(Q' \cap N)f$ is of type I with $\cZ(Q)f$ being an essentially finite index subalgebra of its center. Because $(Q' \cap N)f$ is of type I, we can choose a projection $p_1 \in (Q'\cap N)f$ such that $p_1(Q'\cap N)p_1=\cZ(Q'\cap N)p_1$. Since $\cZ(Q)p_1\subset \cZ(Q'\cap N)p_1$ is an inclusion of abelian von Neumann algebras of essentially finite index, we finally find a nonzero projection $p \in \cZ(Q' \cap N)p_1$ such that $\cZ(Q) p = \cZ(Q'\cap N)p = p (Q' \cap N) p$.
\end{proof}

In the following lemma, we prove that virtual cores behave well w.r.t.\ Popa's intertwining-by-bimodules. This will be crucial in proving that class $\cC$ is stable under extensions. The proof is almost identical to the proof of \cite[Proposition 3.6]{chifan-ioana-kida;arbitrary-actions-central-quotients-braid-groups}. We denote $P^n = M_n(\C) \ot P$.

\begin{lemma}\label{lemma:lem2}
Let $(M,\tau)$ be a tracial von Neumann algebra. Let $A \subset eMe$ and $N \subset M$ be von Neumann subalgebras. Assume that $A \subset eMe$ is a virtual core subalgebra and that $A \prec N$. Then we can choose nonzero projections $p \in N^n$, $z \in \cZ(A)$, an injective normal $\star$-homomorphism $\theta : Az \recht pN^n p$ and a partial isometry $v \in (M_{1,n}(\C) \ot M)p$ satisfying
\begin{itemize}
\item $vv^\star = z$ and $E_{pN^n p}(v^\star v)$ has support projection $p$.
\item $a v = v \theta(a)$ for all $a \in A z$.
\item $\theta(Az)$ is a virtual core subalgebra of $p N^n p$.
\end{itemize}
\end{lemma}
\begin{proof}
Since $A \prec N$, we can choose a projection $p_1 \in N^n$, a normal $\star$-homomorphism $\theta : A \recht p_1 N^n p_1$ and a nonzero partial isometry $w \in e(M_{1,n}(\C) \ot M)p_1$ satisfying $a w = w \theta(a)$ for all $a \in A$. Write $z_1 = w w^\star$ and note that $z_1 \in A' \cap eMe = \cZ(A)$. Note that the restriction of $\theta$ to $A z_1$ is injective.

Denote $q_1 = w^\star w$. We may assume that the support projection of $E_{p_1 N^n p_1}(q_1)$ equals $p_1$. Since $A' \cap eMe = \cZ(A)$, the inclusions $\theta(A z_1) \subset p_1 N^n p_1 \subset p_1 M^n p_1$ and the projection $q_1$ satisfy the assumptions of \cref{lemma:lem1}. So by \cref{lemma:lem1}, we can choose a nonzero projection $p \in p_1 N^n p_1$ such that
$$(\theta(Az_1)p)' \cap p N^n p = \cZ(\theta(Az_1)) p \; .$$
We define $v$ as the polar part of $w p$ and replace $\theta$ by $\theta(\cdot) p$. Put $z = vv^\star$, then we still have that $z \in \cZ(A)$ and that the restriction of $\theta$ to $Az$ is injective. We now also have that $\theta(Az)' \cap p N^n p =  \cZ(\theta(Az))$. Write $q = v^\star v$. By construction, the support projection of $E_{p N^n p}(q)$ equals $p$.

Define $P = \cN_{p N^n p}(\theta(Az))\dpr$. Let $u \in \cN_{z M z}(Az)$ and define $x := E_{p N^n p}(v^\star u v)$. We prove that $x$ belongs to $P$. Define the automorphism $\al \in \Aut(\theta(Az))$ given by $\al = \theta \circ (\Ad u) \circ \theta^{-1}$. By construction, $x a = \al(a) x$ for all $a \in \theta(Az)$. Since the relative commutant of $\theta(Az)$ in $p N^n p$ equals $\cZ(\theta(Az))$, it follows from \cite[Remark 2.4]{jones-popa;properties-MASA-factors} that indeed $x \in P$.

Define $P_0 = \cN_{zMz}(Az)\dpr$. Because $z \in \cZ(A)$, it follows from \cite[Proof of 2.1]{jones-popa;properties-MASA-factors} that $P_0 = z \cN_{eMe}(A)\dpr z$. Since $A \subset e Me$ is a virtual core subalgebra, we conclude that $P_0 \subset z M z$ is essentially of finite index. Thus $v^\star P_0 v \subset q M^n q$ is essentially of finite index. Define $P_1 \subset p N^n p$ as the von Neumann subalgebra generated by $E_{p N^n p}(v^\star P_0 v)$. By \cite[Lemma 2.3]{chifan-ioana-kida;arbitrary-actions-central-quotients-braid-groups}, the inclusion $P_1 \subset p N^n p$ is essentially of finite index. In the previous paragraph, we have proved that $P_1 \subset P$. So a fortiori, $P \subset p N^n p$ is essentially of finite index. So we have proved that $\theta(Az) \subset p N^n p$ is a virtual core subalgebra.
\end{proof}

We can now prove that the class $\cC$ is stable under extensions.

\begin{proposition}\label{proposition:class-C-stable-ext}
The class $\cC$ is stable under extensions.
\end{proposition}
\begin{proof}
Assume that $\Gamma_1 \lhd \Gamma$ with $\Gamma/\Gamma_1 \cong \Gamma_2$ and that $\Gamma_1,\Gamma_2 \in \cC$. We have to prove that $\Gamma \in \cC$. Choose a trace preserving cocycle action $\Gamma \actson (B,\tau)$. Write $M = B \rtimes \Gamma$ and let $A \subset eMe$ be an amenable virtual core subalgebra. We have to prove that $A \prec B$. Write $N = B \rtimes \Gamma_1$. We can then view $M = N \rtimes \Gamma_2$ for some cocycle action $\Gamma_2 \actson N$. Because $\Gamma_2$ belongs to $\cC$, we have that $A \prec N$. Choose projections $z \in \cZ(A)$, $p \in N^n$, a partial isometry $v \in (M_{1,n}(\C) \ot M)p$ and a $\star$-homomorphism $\theta : A z \recht p N^n p$ satisfying the conclusions of \cref{lemma:lem2}.

Since $N^n = (M_n(\C) \ot B) \rtimes \Gamma_1$ for some cocycle action of $\Gamma_1$ on $M_n(\C) \ot B$, and because $\Gamma_1 \in \cC$, it follows that $\theta(Az) \prec M_n(\C) \ot B$ inside $N^n$. Since the support of $E_{p N^n p}(v^\star v)$ equals $p$, this intertwining can be combined with the intertwining given by $v$ and we obtain the conclusion that $A \prec B$ inside $M$.
\end{proof}

Recall that two groups $\Gamma,\Lambda$ are called \emph{commensurable} if they admit isomorphic finite index subgroups.

\begin{proposition}\label{proposition:class-C-stable-commen}
The class $\cC$ is stable under commensurability.
\end{proposition}
\begin{proof}
Let $\Lambda < \Gamma$ be a finite index subgroup. We have to prove that $\Lambda \in \cC$ if and only if $\Gamma \in \cC$. Write $n = [\Gamma: \Lambda]$.

First assume that $\Gamma \in \cC$. Let $\Lambda \actson (B,\tau)$ be a trace preserving cocycle action. Write $M = B \rtimes \Lambda$ and let $A \subset eMe$ be an amenable virtual core subalgebra. Define $B_1 = \ell^\infty(\Gamma/\Lambda) \ot B$ and consider the induced cocycle action $\Gamma \actson B_1$. By construction, $B_1 \rtimes \Gamma \cong M_n(\C) \ot (B \rtimes \Lambda)$. In this way, we can view $M_n(\C) \ot A$ as an amenable virtual core subalgebra of a corner of $B_1 \rtimes \Gamma$. Since $\Gamma \in \cC$, it follows that $M_n(\C) \ot A \prec B_1$ and thus, $A \prec B$.

Conversely assume that $\Lambda \in \cC$. Let $\Gamma \actson (B,\tau)$ be a trace preserving cocycle action. Write $M = B \rtimes \Gamma$ and let $A \subset eMe$ be an amenable virtual core subalgebra. Write $N = B \rtimes \Lambda$. Since $\Lambda < \Gamma$ has finite index, we have $A \prec N$. We apply \cref{lemma:lem2} and find the virtual core subalgebra $\theta(Az) \subset p N^n p$. Since $\Lambda \in \cC$, it follows that $\theta(Az) \prec B$, but then also $A \prec B$, as in the final paragraph of the proof of \cref{proposition:class-C-stable-ext}.
\end{proof}

\section{A non-isomorphism result for almost periodic crossed products}\label{section:non-isomorphism}

We prove the following general rigidity theorem for crossed products of groups in the class $\cC$ acting by state preserving automorphisms of an amenable factor with almost periodic state.
This will be a crucial ingredient to distinguish between type $\III$ Bernoulli crossed products.

\begin{theorem}\label{thm.non-iso-class-C-type-III}
For $i=0,1$, let $\Lambda_i$ be icc groups in the class $\cC$. Let $(P_i,\phi_i)$ be amenable factors equipped with normal faithful almost periodic states having a factorial discrete decomposition. Let $\Lambda_i \actson (P_i,\phi_i)$ be outer, state preserving actions such that the crossed products $P_i \rtimes \Lambda_i$ are full.

Then the following two statements are equivalent.
\begin{enumerate}[\upshape(i)]
\item The crossed products $P_i \rtimes \Lambda_i$ are isomorphic.
\item The groups $\Lambda_0, \Lambda_1$ are isomorphic, the point spectra of $\Delta_{\phi_0}$ and $\Delta_{\phi_1}$ coincide and there exists a projection $p \in (P_1)_{\phi_1}$, equal to $1$ if the $\phi_i$ are traces, such that $\Lambda_0 \actson (P_0,\phi_0)$ is cocycle conjugate to the reduced cocycle action $(\Lambda_1 \actson P_1)^p$ through a state preserving isomorphism, modulo the group isomorphism $\Lambda_0 \cong \Lambda_1$.
\end{enumerate}
\end{theorem}
\begin{proof}
Denote by $G_i$ the natural compact groups acting by the modular automorphisms $(\si^i_s)_{s \in G_i}$ on $P_i$ and on $P_i \rtimes \Lambda_i$. Write $N_i = P_i \rtimes G_i$. The discrete decomposition of $P_i \rtimes \Lambda_i$ is then given as $N_i \rtimes \Lambda_i$. Since the groups $\Lambda_i$ have trivial center, it follows from \cref{lemma:outer-action-on-core} that the actions $\Lambda_i \actson N_i$ are outer. In particular, $P_i \rtimes \Lambda_i$ is a full factor with a factorial discrete decomposition. Therefore, $\Sd(P_i \rtimes \Lambda_i)$ equals the point spectrum of $\Delta_{\phi_i}$. In particular, the factor $P_i \rtimes \Lambda_i$ is of type $\II_1$ if $\phi_i$ is a trace and of type $\III$ if $\phi_i$ is not a trace.

Assume first that (ii) holds. In the tracial case, the $\II_1$ factors $P_i \rtimes \Lambda_i$ follow isomorphic. In the non-tracial case, it follows that $P_0 \rtimes \Lambda_0$ is isomorphic to $p(P_1 \rtimes \Lambda_1)p$. But in that case, the factors $P_i \rtimes \Lambda_i$ are of type $\III$ so that also $P_0 \rtimes \Lambda_0 \cong P_1 \rtimes \Lambda_1$. This means that (i) holds.

Assume next that (i) holds and let $\psi : P_0 \rtimes \Lambda_0 \recht P_1 \rtimes \Lambda_1$ be a $\star$-isomorphism. In particular, the factors $P_i \rtimes \Lambda_i$ have the same $\Sd$-invariant, so that the point spectra of $\Delta_{\phi_i}$ coincide. In the case where the $\phi_i$ are traces, because the groups $\Lambda_i$ are icc and belong to the class $\cC$, it follows from \cref{lemma:weakcontainmentisunitaryconjug} that $\psi(P_0)$ is unitarily conjugate to $P_1$. Since the actions $\Lambda_i \actson P_i$ are outer, this means that $\Lambda_0 \cong \Lambda_1$ and that the actions $\Lambda_i \actson P_i$ are cocycle conjugate.

So assume that the $\phi_i$ are not traces. Since the point spectra of $\Delta_{\phi_i}$ coincide, it follows that $G_0 = G_1$ and we denote this compact group as $G$. By \cite[Lemma 4.2]{connes;almost-periodic-states}, $\psi$ is a cocycle conjugacy between the modular automorphism groups $(\si^0_s)_{s \in G}$ and $(\si^1_s)_{s \in G}$ and therefore extends to a $\star$-isomorphism $\Psi : M_0 \recht M_1$ between the crossed products $M_i = (P_i \rtimes \Lambda_i) \rtimes G = N_i \rtimes \Lambda_i$. The assumption that $(P_i,\phi_i)$ is amenable with factorial discrete decomposition means that $N_i$ is the hyperfinite $\II_\infty$ factor. We have seen in the first paragraph of the proof that the actions $\Lambda_i \actson N_i$ are outer. Since the actions $\Lambda_i \actson P_i$ are state preserving, the action of $\Lambda_i$ on $N_i$ equals the identity on $L(G) \subset N_i$.

We claim that $\Psi(N_0)$ and $N_1$ are unitarily conjugate inside $M_1$. Take a projection $p_0 \in L(G)$ of finite trace. Then $\Psi(p_0)$ is a projection of finite trace in the $\II_\infty$ factor $N_1 \rtimes \Lambda_1$. After a unitary conjugacy of $\Psi$, we find a projection $p_1 \in L(G)$ of finite trace such that $\Psi(p_0) \leq p_1$. Since the projections $p_i$ are $\Lambda_i$-invariant, we have
$$p_i M_i p_i = p_i N_i p_i \rtimes \Lambda_i \; .$$
The restriction of $\Psi$ to $p_0 M_0 p_0$ thus yields a $\star$-isomorphism of $p_0 N_0 p_0 \rtimes \Lambda_0$ onto a corner of $p_1 N_1 p_1 \rtimes \Lambda_1$. Because the groups $\Lambda_i$ are icc and belong to the class $\cC$, it follows from \cref{lemma:weakcontainmentisunitaryconjug} that $\Psi(p_0 N_0 p_0)$ is unitarily conjugate to a corner of $N_1$. Since the $N_i$ are $\II_\infty$ factors, the claim follows.

By the claim in the previous paragraph, we can choose a unitary $u \in \cU(M_1)$ such that $u \Psi(N_0) u^\star = \Psi(N_1)$. In particular, $\Lambda_0 \cong \Lambda_1$ and $\Ad u \circ \Psi$ is a cocycle conjugacy between $\Lambda_0 \actson N_0$ and $\Lambda_1 \actson N_1$.

Denote $\Gamma = \hat{G}$ and let $\hat{\sigma}^i$ be the dual, trace scaling action of $\Gamma$ on $M_i = (P_i \rtimes \Lambda_i) \rtimes G$. By construction, $\Psi \circ \hat{\sigma}^0_\gamma = \hat{\sigma}^1_\gamma \circ \Psi$ for all $\gamma \in \Gamma$. Therefore, $\Psi$ further extends to a $\star$-isomorphism $\Psitil : M_0 \rtimes \Gamma \recht M_1 \rtimes \Gamma$ satisfying $\Psitil(u_\gamma) = u_\gamma$ for all $\gamma \in \Gamma$. Write $\Theta = \Ad u \circ \Psitil$. Note that we can view $M_i \rtimes \Gamma$ as $N_i \rtimes (\Lambda_i \times \Gamma)$, and that the actions $\Lambda_i \times \Gamma \actson N_i$ are outer because the action of $\Gamma$ is trace scaling and the action of $\Lambda_i$ is trace preserving and outer. By construction,
$$\Theta(N_0) = N_1 \quad , \quad \Theta(N_0 \rtimes \Lambda_0) = N_1 \rtimes \Lambda_1 \quad\text{and}\quad \Theta(u_\gamma) \in (N_1 \rtimes \Lambda_1) u_\gamma \; .$$
This means that the restriction of $\Theta$ to $N_0$ is a cocycle conjugacy between the actions $\Lambda_i \times \Gamma \actson N_i$, modulo a group isomorphism $\delta : \Lambda_0 \times \Gamma \recht \Lambda_1 \times \Gamma$ satisfying $\delta(\Lambda_0) = \Lambda_1$ and $\delta(e,\gamma) \in \Lambda_1 \times \{\gamma\}$. Since $\Lambda_i$ has trivial center, this means that $\delta(g,\gamma) = (\delta_0(g),\gamma)$ for all $g \in \Lambda_0$, $\gamma \in \Gamma$ and a group isomorphism $\delta_0 : \Lambda_0 \recht \Lambda_1$.

It now follows from \cref{lemma:cocycleconjugatecore} that there exists a projection $p \in (P_1)_{\phi_1}$ such that $\Lambda_0 \actson (P_0,\phi_0)$ is cocycle conjugate to the reduced cocycle action $(\Lambda_1 \actson P_1)^p$ through a state preserving isomorphism, modulo the group isomorphism $\delta_0 : \Lambda_0 \recht \Lambda_1$. So we have proved that (ii) holds.
\end{proof}

\section{Proofs of theorems \ref{thm.A}, \ref{thm.B} and \ref{thm.C}}

Theorems \ref{thm.A} and \ref{thm.B} follow from \cref{thm:cocycle-conjugacy-free-products} and the following result.

\begin{theorem}\label{thm.general-rigidity}
For $i=0,1$, let $(P_i,\phi_i)$ be amenable factors equipped with a normal faithful state. Let $\Lambda_i$ be icc groups in the class $\cC$. Assume that $\phi_0$ is almost periodic and that $P_0^{\Lambda_0} \rtimes \Lambda_0$ is isomorphic with $P_1^{\Lambda_1} \rtimes \Lambda_1$. Then the following holds.
\begin{enumerate}[\upshape(i)]
\item Also $\phi_1$ is almost periodic and $\Gamma(P_0,\phi_0) = \Gamma(P_1,\phi_1)$.
\item The groups $\Lambda_i$ are isomorphic.
\item If $\Lambda_0$ is a direct product of two icc groups in the class $\cC$, then the actions $\Lambda_i \actson P_i^{\Lambda_i}$ are conjugate through a state preserving isomorphism.
\end{enumerate}
\end{theorem}
\begin{proof}
Since $\Lambda_i$ belongs to the class $\cC$, certainly $\Lambda_i$ is nonamenable. By \cref{lemma.full-bernoulli}, the factors $P_i^{\Lambda_i} \rtimes \Lambda_i$ are full and their $\tau$-invariant is the weakest topology that makes the map $t \mapsto \sigma_t^{\phi_i}$ continuous. Since the factors $P_i^{\Lambda_i} \rtimes \Lambda_i$ are isomorphic, they have the same $\tau$-invariant and we conclude that (i) holds.

Write $\vphi_i = \phi_i^{\Lambda_i}$. By \cref{lemma:centralizerfactor}, the factors $(P_i^{\Lambda_i},\vphi_i)$ have a factorial discrete decomposition. By \cref{lemma:bernoulli-outer}, the actions $\Lambda_i \actson^{\alpha_i} P_i^{\Lambda_i}$ are outer. Above, we already mentioned that the crossed products are full factors. By \cref{thm.non-iso-class-C-type-III}, (ii) holds and we find a projection $p \in (P_0^{\Lambda_0})_{\vphi_0}$ such that the action $\Lambda_1 \actson P_1^{\Lambda_1}$ is cocycle conjugate to the reduction of $\Lambda_0 \actson P_0^{\Lambda_0}$ by $p$, modulo the isomorphism $\Lambda_0 \cong \Lambda_1$, through a generalized $1$-cocycle $(w_g)_{g \in \Lambda_0}$ for the action $\Lambda_0 \actson P_0^{\Lambda_0}$ having support projection $p$.

Assume now that $\Lambda_0$ is moreover a direct product of two nonamenable groups. We apply \cref{cor.special-case} to the generalized $1$-cocycle $(w_g)_{g \in \Lambda_0}$. Since the action $\Lambda_1 \actson P_1^{\Lambda_1}$ has no nontrivial finite dimensional globally invariant subspaces, it follows from \cref{cor.special-case} that $w_g = \chi(g) w \al^0_g(w^\star)$ for all $g \in \Lambda_0$, where $\chi : \Lambda_0 \recht \T$ is a character and $w \in (P_0^{\Lambda_0})_{\vphi_0,\lambda}$ satisfies $ww^\star = p$ and $w^\star w = 1$. Conjugating with $w$, we conclude that the actions $\Lambda_i \actson P_i^{\Lambda_i}$ are conjugate through a state preserving isomorphism.
\end{proof}

We finally prove \cref{thm.C}.

\begin{proof}[Proof of \cref{thm.C}]
For $i=0,1$, let $\Lambda_i$ be icc groups in the class $\cC$ without nontrivial central sequences. Let $(P_i,\phi_i)$ be amenable factors with a normal faithful almost periodic state. Obviously, if $\Lambda_0 \cong \Lambda_1$ and $(P_0,\phi_0) \cong (P_1,\phi_1)$ through a state preserving isomorphism, then the von Neumann algebras $P_i^{\Lambda_i} \rtimes (\Lambda_i \times \Lambda_i)$ are isomorphic. Assume conversely that the $P_i^{\Lambda_i} \rtimes (\Lambda_i \times \Lambda_i)$ are isomorphic.

Since $\Lambda_i$ has trivial center, every nontrivial element of $\Lambda_i \times \Lambda_i$ moves infinitely many elements of $\Lambda_i$. By \cref{lemma:bernoulli-outer}, the Bernoulli action of $\Lambda_i \times \Lambda_i$ on $(P_i,\phi_i)^{\Lambda_i}$ is outer. By \cref{lemma:centralizerfactor}, the almost periodic factor $(P_i,\phi_i)^{\Lambda_i}$ has a factorial discrete decomposition. Finally, since $\Lambda_i$ has no nontrivial central sequences, there exists a finite subset $\cF_i \subset \Lambda_i$ whose centralizer is trivial. Then, the stabilizer of $\{e\} \cup \cF_i$ under the action of $\Lambda_i \times \Lambda_i$ on $\Lambda_i$ is trivial. Thus, by \cref{lemma.full-bernoulli}, the crossed products $P_i^{\Lambda_i} \rtimes (\Lambda_i \times \Lambda_i)$ are full.

Proceeding exactly as in the proof of \cref{thm.general-rigidity}, using in particular \cref{cor.special-case} and the weak mixing of $\Lambda_1 \times \Lambda_1 \actson P_1^{\Lambda_1}$, we find an isomorphism $\delta : \Lambda_0 \times \Lambda_0 \recht \Lambda_1 \times \Lambda_1$ and a state preserving isomorphism $\psi : P_0^{\Lambda_0} \recht P_1^{\Lambda_1}$ satisfying $\psi \circ \al^0_g = \al^1_{\delta(g)} \circ \psi$ for all $g \in \Lambda_0 \times \Lambda_0$. Here, $\al^i$ denotes the Bernoulli action.

We continue with an argument from \cite[Proof of Theorem 5.4]{popa-vaes;strong-rigidity-generalized-bernoulli-actions}. Write $\Delta_i = \{(g,g) \mid g \in \Lambda_i\}$. If $\Sigma < \Lambda_i \times \Lambda_i$ is a subgroup such that $\Sigma \cdot g$ is infinite for all $g \in \Lambda_i$, then the multiples of $1$ are the only $(\al^i_g)_{g \in \Sigma}$-invariant elements of $P_i^{\Lambda_i}$. Denote by $\pi_e : P_i \recht P_i^{\Lambda_i}$ the embedding as the $e$-th tensor factor. Since every element of $\pi_e(P_0)$ is $\Delta_0$-invariant, it follows that there exists a $g \in \Lambda_1$ such that $\delta(\Delta_0) \cdot g$ is finite. Composing $\psi$ with $\al^1_{(e,g)}$ and $\delta$ with $\Ad (e,g)$, we may assume that $g = e$. So we find a finite index subgroup $\Delta'_0 < \Delta_0$ such that $\delta(\Delta'_0) \subset \Delta_1$.

But then, every element of $\psi^{-1}(\pi_e(P_1))$ is $\Delta'_0$-invariant. Since $\Lambda_0$ is icc, the sets $\Delta_0 \cdot g$ are infinite for all $g \neq e$. So also $\Delta'_0 \cdot g$ is infinite for all $g \neq e$. Therefore, all $\Delta'_0$-invariant elements of $P_0^{\Lambda_0}$ belong to $\pi_e(P_0)$. We conclude that $\psi^{-1}(\pi_e(P_1)) \subset \pi_e(P_0)$. Since the elements of $\delta^{-1}(\Delta_1)$ leave every element of $\psi^{-1}(\pi_e(P_1))$ fixed, it follows that $\delta^{-1}(\Delta_1) \subset \Delta_0$. In particular, every element of $\psi(\pi_e(P_0))$ is $\Delta_1$-invariant and thus belongs to $\pi_e(P_1)$. We have proved that $\psi(\pi_e(P_0)) = \pi_e(P_1)$. Then also $\delta(\Delta_0) = \Delta_1$ and the theorem is proved.
\end{proof}

\section*{Appendix: Popa's cocycle superrigidity for Connes-St{\o}rmer Bernoulli actions}
\refstepcounter{section}

Let $(M,\vphi)$ be a von Neumann algebra equipped with a normal faithful almost periodic state and an action $(\al_g)_{g \in \Lambda}$ of a group $\Lambda$ by state preserving automorphisms. A \emph{generalized $1$-cocycle} for $(\al_g)_{g \in \Lambda}$ is a family of partial isometries $w_g \in M_\vphi$ such that $w_g w_g^\star = p$, independently of $g \in \Lambda$, and $w_g^\star w_g = \al_g(p)$, satisfying
$$w_{gh} = \Om(g,h) w_g \al_g(w_h) \quad\text{for all}\;\; g,h \in \Lambda \; ,$$
where $\Om : \Lambda \times \Lambda \recht \T$ is a scalar $2$-cocycle. We call $p$ the support projection of $(w_g)_{g \in \Lambda}$. Note that $p \in M_\vphi$.

In \cite{popa;rigidity-non-commutative-bernoulli}, Sorin Popa proved a cocycle superrigidity theorem for the Connes-St{\o}rmer Bernoulli actions of property (T) groups and of $w$-rigid groups, i.e.~groups admitting an infinite normal subgroup with the relative property (T). In \cite{popa;cocycle-superrigidity-malleable-actions}, cocycle superrigidity was established for the `classical' (commutative) Bernoulli actions of $w$-rigid groups with arbitrary countable target groups, or even target groups in Popa's class $\mathcal{U}_\text{fin}$. Both in \cite{popa;rigidity-non-commutative-bernoulli,popa;cocycle-superrigidity-malleable-actions}, the rigidity is provided by Kazhdan's (relative) property (T). In \cite{popa;superrigidity-malleable-actions-spectral-gap}, Popa discovered that cocycle superrigidity can also be proved using his spectral gap rigidity, typically for groups $\Lambda$ that arise as the direct product of an infinite group and a nonamenable group. In \cite{popa;superrigidity-malleable-actions-spectral-gap}, only the commutative Bernoulli actions are treated. Mutatis mutandis, the methods of \cite{popa;superrigidity-malleable-actions-spectral-gap} and \cite{popa;rigidity-non-commutative-bernoulli} can be combined to obtain cocycle superrigidity for Connes-St{\o}rmer Bernoulli actions of product groups with general target groups. For the sake of completeness, we include a complete argument in this appendix and benefit from this occasion to state and prove the most general result that can be obtained along these lines.

\begin{theorem}\label{thm.cocycle-cs}
Let $(P,\phi)$, $(Q,\vphi)$ be von Neumann algebras with an almost periodic normal faithful state. Let $\Lambda \actson^\alpha (Q,\vphi)$ be a state preserving action. Assume that $\Lambda$ also acts on the countable set $I$ and consider the (generalized) Bernoulli action $\Lambda \actson (M,\vphi) = (P,\phi)^I$. We also denote this action by $\al$, as well as the diagonal action
$$\al : \Lambda \actson M \ovt Q \; .$$
Assume that $\Lambda$ is generated by the commuting subgroups $\Gamma$ and $\Sigma$ satisfying the following two properties: the action $\Gamma \actson I$ has no invariant mean, and the action $\Sigma \actson I$ has infinite orbits.

Let $p \in (M \ovt Q)_\vphi$ be a projection and $w_g \in (M \ovt Q)_\vphi$ a generalized $1$-cocycle with support projection $p$.
Then, $p = \sum_k p_k$ for projections $p_k \in (M \ovt Q)_\vphi$ such that $p_k w_g = w_g \al_g(p_k)$ for all $g \in \Lambda$ and the generalized $1$-cocycle $(p_k w_g)_{g \in \Lambda}$ is cohomologous to a generalized $1$-cocycle taking values in an amplification of $Q_\vphi$.

More precisely, there exist Hilbert spaces $H_k$, positive numbers $\lambda_k > 0$, projections $q_k \in Q_\vphi \ovt B(H_k)$ with $(\vphi \ot \Tr)(q_k) < \infty$ and elements $v_k \in (M \ovt Q)_{\vphi,\lambda_k} \ovt \Hbar$ such that
$$v_k v_k^\star = p_k \;\;\text{and}\;\; v_k^\star v_k = 1 \ot q_k \quad , \quad v_k^\star \, w_g \, \al_g(v_k) = 1 \ot W_{k,g}$$
where $(W_{k,g})_{g \in \Lambda}$ is a generalized $1$-cocycle for the (amplified) action $\al_g$ on $Q_\vphi \ovt B(H_k)$ with support projection $q_k$.
\end{theorem}

In the special case where $Q = \C 1$ and the twisted action $\Ad w_g \circ \al_g$ is assumed to have a trivial fixed point algebra, much more can be said, see \cref{cor.special-case}.

To prove \cref{thm.cocycle-cs}, we make use of Ioana's variant \cite{ioana;rigidity-results-wreath-product} of Popa's malleable deformation \cite{popa;rigidity-non-commutative-bernoulli} of the Connes-St{\o}rmer Bernoulli action, defined as follows. Consider $(\Ptil,\phi) = (P,\phi)\star(L\Z,\tau)$ and $(\Mtil,\vphi) = (\Ptil,\phi)^I$. Denote by $(u_n)_{n \in \Z}$ the canonical unitary operators in $L\Z$ and denote by $h \in L\Z$ the selfadjoint element with spectrum $[-\pi,\pi]$ satisfying $u_1 = \exp(ih)$. Define $u_t = \exp(ith)$ for all $t \in \mathbb{R}$. Equip $\Mtil$ with the one-parameter group of state preserving automorphisms $(\theta_t)_{t \in \mathbb{R}}$ given as the infinite tensor product $(\Ad u_t)^I$. Define the period $2$ automorphism $\gamma$ of $\Mtil$ as the infinite tensor product of the automorphism of $\Ptil$ satisfying $x \mapsto x$ for all $x \in P$ and $u_1 \mapsto u_{-1}$. Note that $\gamma \circ \theta_t = \theta_{-t} \circ \gamma$. We still denote by $\theta_t$ and $\gamma$ the automorphisms of $\Mtil \ovt Q$ that act as the identity on $Q$. We need the following variant of \cite[Lemma 2.10]{popa;cocycle-superrigidity-malleable-actions}.

\begin{lemma}\label{lem.weak-mixing}
Assume that the action $\Sigma \actson I$ has infinite orbits.
\begin{enumerate}
\item If $(w_h)_{h \in \Sigma}$ is a generalized $1$-cocycle for the action $\al_h$ on $(M \ovt Q)_\vphi$ with support projection $p$ and if $x \in p(\Mtil \ovt Q)p$ satisfies $x = w_h \al_h(x) w_h^\star$ for all $h \in \Sigma$, then $x \in M \ovt Q$.
\item Assume that $H$ is a Hilbert space, $q_1,q_2 \in Q_\vphi \ovt B(H)$ are projections with $(\vphi \ot \Tr)(q_i) < \infty$ and $(W_{i,h})_{h \in \Sigma}$ are generalized $1$-cocycles for the (amplified) action $\al_h$ on $Q_\vphi \ovt B(H)$ with support projection $q_i$ and the same scalar $2$-cocycle $\Om$. If $x \in M \ovt q_1(Q \ovt B(H))q_2$ satisfies $x = (1 \ot W_{1,h}) \al_h(x) (1 \ot W_{2,h}^\star)$ for all $h \in \Sigma$, then $x \in 1 \ovt Q \ovt B(H)$.
\end{enumerate}
\end{lemma}
\begin{proof}
1.\ Denote by $J$ the set of all nonempty subsets of $I$. For every $\cF \in J$, we denote by $K_\cF$ the $\|\fdot\|_\vphi$-closed linear span of $M (\Ptil \ominus P)^\cF M \ot Q$ inside $L^2(\Mtil \ovt Q)$. Note that $L^2((\Mtil \ominus M) \ovt Q)$ is the orthogonal direct sum of the closed subspaces $K_\cF$, $\cF \in J$. Denote by $p_\cF$ the orthogonal projection of $L^2(\Mtil \ovt Q)$ onto $K_\cF$.

Take $x \in p(\Mtil \ovt Q)p$ satisfying $x = w_h \al_h(x) w_h^\star$ for all $h \in \Sigma$. Define $\xi \in \ell^2(J)$ given by $\xi(\cF) = \|p_\cF(x)\|_\vphi$. Since $K_\cF$ is an $(M \ovt Q)$-bimodule, we have for all $h \in \Sigma$ that
$$(h \cdot \xi)(\cF) = \|p_{h^{-1} \cdot \cF}(x)\|_\vphi = \|p_{\cF}(\al_h(x))\|_\vphi = \|p_{\cF}(w_h \, \al_h(x) \, w_h^\star)\|_\vphi = \|p_\cF(x)\|_\vphi = \xi(\cF) \; .$$
Since $\Sigma \actson I$ has infinite orbits, there are no nonzero $\Sigma$-invariant vectors in $\ell^2(J)$. So $p_\cF(x) = 0$ for all $\cF \in J$. This means that $x \in M \ovt Q$.

2.\ We decompose $L^2((M \ominus \C 1) \ovt Q \ovt B(H))$ as the orthogonal direct sum of the subspaces $\cL_\cF$, $\cF \in J$, defined as the $\|\fdot\|_\vphi$-closed linear span of $(P \ominus \C 1)^\cF \ot Q \ot \HS(H)$, where $\HS(H)$ denotes the set of Hilbert-Schmidt operators on $H$. We then reason in the same way as in 1.
\end{proof}

\begin{proof}[Proof of \cref{thm.cocycle-cs}]
{\bf Step 1:} we prove the following spectral gap property. There exists a $\kappa > 0$ and $g_1,\ldots,g_n \in \Gamma$ such that for all $\xi \in L^2((\Mtil \ominus M) \ovt Q)$, $h \in \Sigma$ and $\mu_i \in \T$, we have
\begin{equation}\label{eq.claim}
\|\xi\|_\vphi^2 \leq \kappa \sum_{i=1}^n \|\mu_i\, \xi - w_{g_i} \, \al_{g_i}(\xi) \, \al_h(w_{g_i}^\star) \|_\vphi^2 \; .
\end{equation}
Denote by $J$ the set of all nonempty subsets of $I$. By \cref{lem.mean}, we can take a $\kappa > 0$ and $g_1,\ldots,g_n \in \Gamma$ such that
\begin{equation}\label{eq.step1}
\|\eta\|_2^2 \leq \kappa \sum_{i=1}^n \|\eta - g_i \cdot \eta\|_2^2 \quad\text{for all}\;\; \eta \in \ell^2(J) \; .
\end{equation}
We choose $\kappa$ such that also $\kappa \geq 1/n$. To prove \eqref{eq.claim}, fix $h \in \Sigma$ and $\xi \in L^2((\Mtil \ominus M) \ovt Q)$. Put $\xi_0 = p \xi \al_h(p)$. The left hand side of \eqref{eq.claim} equals $\|\xi - \xi_0\|_\vphi^2 + \|\xi_0\|_\vphi^2$, while the right hand side equals
$$\kappa n \|\xi-\xi_0\|_\vphi^2 + \kappa \sum_{i=1}^n \|\mu_i\, \xi_0 - w_{g_i} \, \al_{g_i}(\xi_0) \, \al_h(w_{g_i}^\star) \|_\vphi^2 \; .$$
Since $\kappa n \geq 1$, we may from the beginning assume that $\xi = \xi_0$, i.e.\ that $\xi = p \xi \al_h(p)$.

As in the proof of \cref{lem.weak-mixing}.1, we decompose $L^2((\Mtil \ominus M) \ovt Q)$ as the orthogonal direct sum of the closed subspaces $K_\cF$, with orthogonal projection $p_\cF$ onto $K_\cF$. Define $\eta \in \ell^2(J)$ given by $\eta(\cF) = \|p_\cF(\xi)\|_\vphi$. Note that $\|\xi\|_\vphi = \|\eta\|_2$. Since $K_\cF$ is an $(M \ovt Q)$-bimodule, and $\xi = p\xi \alpha_h(p)$, we have for all $g \in \Gamma$ that
$$(g \cdot \eta)(\cF) = \|p_{g^{-1} \cdot \cF}(\xi)\|_\vphi = \|p_{\cF}(\al_g(\xi))\|_\vphi = \|p_{\cF}(w_g \, \al_g(\xi) \, \al_h(w_g^\star))\|_\vphi \; .$$
Therefore, for all $\mu \in \T$, we have
\begin{align*}
\|\eta - g \cdot \eta\|^2_2 &= \sum_{\cF \in J} \bigl| \, \|p_{\cF}(\mu \xi)\|_\vphi - \|p_{\cF}(w_g \, \al_g(\xi) \, \al_h(w_g^\star))\|_\vphi \, \bigr|^2 \\
& \leq \sum_{\cF \in J} \|p_\cF( \mu \xi - w_g \, \al_g(\xi) \, \al_h(w_g^\star))\|_\vphi^2 = \| \mu \xi - w_g \, \al_g(\xi) \, \al_h(w_g^\star)\|_\vphi^2 \; .
\end{align*}
Taking $\mu = \mu_i$, $g=g_i$ and summing over $i$, \eqref{eq.claim} follows from \eqref{eq.step1}.

{\bf Step 2:} we prove that there exists a $t_0 > 0$ such that
\begin{equation}\label{eq.aim2}
\varphi(\theta_t(w_h) \, w_h^\star) \geq \frac{1}{2} \varphi(p) \quad\text{for all}\;\; h \in \Sigma, 0 \leq t \leq t_0 \; .
\end{equation}
Since $\Sigma$ and $\Gamma$ commute inside $\Lambda$, the formula
$$\kappa(g,h) = \Om(g,h) \, \overline{\Om}(h,g)$$
defines a bicharacter $\kappa : \Gamma \times \Sigma \recht \T$. The $1$-cocycle relation implies that
$$\kappa(g,h) \, w_h = w_g \, \al_g(w_h) \, \al_h(w_g^\star) \quad\text{for all}\;\; h \in \Sigma, g \in \Gamma \; .$$
Applying $\theta_{t}$, it follows that
$$\| \kappa(g,h)\, \theta_t(w_h) - w_g \, \al_g(\theta_t(w_h)) \, \al_h(w_g^\star) \|_\vphi \leq 2 \|w_g - \theta_t(w_g)\|_\vphi \; .$$
We write $\xi_h = \theta_t(w_h) - E_{M \ovt Q}(\theta_t(w_h))$ and conclude that
$$\| \kappa(g,h) \, \xi_h - w_g \, \al_g(\xi_h) \, \al_h(w_g)^\star \|_\vphi \leq 2 \|w_g - \theta_t(w_g)\|_\vphi \; .$$
In combination with \eqref{eq.claim}, it follows that
$$\| \xi_h \|_\vphi^2 \leq 4 \kappa \sum_{i=1}^n \|w_{g_i} - \theta_t(w_{g_i})\|_\vphi^2 \; .$$
Fix $t_0 > 0$ such that the right hand side of the previous expression is smaller than $\vphi(p)/2$ for all $0 \leq t \leq t_0$. It follows that $\|\xi_h\|_\vphi^2 \leq \vphi(p)/2$ and hence
\begin{equation}\label{eq.almost-2}
\|E_{M \ovt Q}(\theta_t(w_h))\|_\vphi^2 \geq \frac{1}{2} \vphi(p) \quad\text{for all}\;\; h \in \Sigma, 0 \leq t \leq t_0 \; .
\end{equation}
Let $x \in M \ovt Q$, and write $x = \sum_{\cF} x_\cF$, where $x_\cF$ is the orthogonal projection of $x$ onto the $\|\fdot\|_\vphi$-closed linear span of $(P \ominus \mathbb{C})^\cF \ot Q$ inside $L^2(M \ovt Q)$, and $\cF$ runs over all finite subsets of $I$. Denoting $\rho(t) = |\vphi(u_t)|^2$, we have $E_{M\ovt Q}(\theta_t(x)) = \sum_\cF \rho(t)^{|\cF|} x_\cF$, and hence
\begin{align*}
 \vphi(x^\star E_{M \ovt Q}(\theta_t(x)))
  = \sum_{\cF} \rho(t)^{|\cF|} \| x_\cF \|_2^2
  \geq \sum_{\cF} \rho(t)^{2 |\cF|} \| x_\cF \|_2^2
  = \| E_{M \ovt Q}(\theta_t(x)) \|_2^2.
\end{align*}
We obtain for all $x \in (M \ovt Q)_\vphi$ the following transversality inequality in the sense of \cite[Lemma 2.1]{popa;superrigidity-malleable-actions-spectral-gap}:
$$\vphi(\theta_t(x) x^\star) = \vphi(x^\star \theta_t(x)) \geq \|E_{M \ovt Q}(\theta_t(x))\|_\vphi^2 \; .$$
In combination with \eqref{eq.almost-2}, we have proved that \eqref{eq.aim2} holds.

{\bf Step 3:} there exists a partial isometry $V \in (\Mtil \ovt Q)_\varphi$ and a nonzero projection $p_0 \in (M \ovt Q)_\varphi$ satisfying $p_0 \leq p$, $V^\star V = p_0$, $V V^\star = \theta_1(p_0)$ and
\begin{equation}\label{eq.aim3}
V \, w_h = \theta_1(w_h) \, \al_h(V) \quad\text{for all}\;\; h \in \Sigma \; .
\end{equation}
Take a positive integer $r$ such that $2^{-r} \leq t_0$. Put $t = 2^{-r}$. Define $K \subset (\Mtil \ovt Q)_\vphi$ as the $\|\fdot\|_2$-closed convex hull of $\{\theta_t(w_h) w_h^\star \mid h \in \Sigma\}$. Define $V_1$ as the unique element of minimal $\|\fdot\|_2$ in $K$. By \eqref{eq.aim2}, we have that $\vphi(V_1) \geq \vphi(p)/2$, so that $V_1$ is nonzero. For all $h,h' \in \Sigma$, we have
$$\theta_t(w_h) \, \al_h(x) \, w_h^\star = y \quad\text{where}\;\; x = \theta_t(w_{h'}) w_{h'}^\star \;\; , \;\; y = \theta_t(w_{hh'}) w_{hh'}^\star \; .$$
Therefore, $K$ is invariant under $\pi_h : K \recht K : \pi_h(x) = \theta_t(w_h) \, \al_h(x) \, w_h^\star$. Since $x = \theta_t(p) x p$ for all $x \in K$, the maps $\pi_h$ are isometric. Therefore, $\pi_h(V_1) = V_1$ for all $h \in \Sigma$. This means that $\theta_t(w_h) \, \al_h(V_1) = V_1 \, w_h$ for all $h \in \Sigma$. Similarly, $K$ is invariant under $x \mapsto \theta_t(\gamma(x))^\star$, so that $V_1^\star = \theta_t(\gamma(V_1))$.

Taking the polar decomposition of $V_1$, we may assume that $V_1$ is a nonzero partial isometry with $V_1^\star V_1 \leq p$ and $V_1 V_1^\star \leq \theta_t(p)$. Since $\al_h(V_1^\star V_1) = w_h^\star \, V_1^\star V_1 \, w_h$ for all $h \in \Sigma$, \cref{lem.weak-mixing}.1 implies that $p_0 := V_1^\star V_1 \in (M \ovt Q)_\varphi$. Then $V_1 V_1^\star = \theta_t(p_0)$ and it follows that $V_2 := \theta_t(V_1 \, \gamma(V_1^\star))$ is a partial isometry in $(\Mtil \ovt Q)_\varphi$ with left support $\theta_{2t}(p_0)$ and right support $p_0$, satisfying $\theta_{2t}(w_h) \, \al_h(V_2) = V_2 \, w_h$ for all $h \in \Sigma$. Continuing the same procedure up to $r$ steps, we arrive at \eqref{eq.aim3}.

{\bf Step 4:} adding the subscript $0$ to indicate the linear span of all eigenvectors of the modular automorphism group of $\varphi$, we prove that there is no sequence $h_n \in \Sigma$ satisfying
\begin{equation}\label{eq.absurd}
\begin{aligned}
h_n \cdot i &\recht \infty \;\;&&\text{for all}\;\; i \in I,\\
\text{and}\quad\|(\vphi \ot \id)((x^\star \ot 1) &w_{h_n} (\al_{h_n}(y) \ot 1))\|_\vphi \recht 0 \;\;&&\text{for all}\;\; x,y \in M_0 \; .
\end{aligned}
\end{equation}
Assume by contradiction that $(h_n)$ is such a sequence. We claim that
\begin{equation}\label{eq.absurd-2}
\|E_{M \ovt Q}((x^\star \ot 1) \theta_1(w_{h_n}) (\al_{h_n}(y) \ot 1))\|_\vphi \recht 0 \quad\text{for all}\;\; x,y \in \Mtil_0 \; .
\end{equation}
Denote by $W \subset \Ptil$ the set of ``reduced words'' in $\Ptil = P \star L\Z$, i.e.\ products of factors alternatingly belonging to $P_0 \ominus \C 1$ and $\{u_n \mid n \in \Z \setminus \{0\}\}$. We also consider $1 \in W$ as the ``empty word''. We have $W = W_1 \sqcup W_2$ where the elements $w \in W_1$ belong to $\theta_1(P) P$ and the elements $w \in W_2$ are orthogonal to $\theta_1(P) P$. We finally denote, for $\cF \subset I$, by $W^\cF \subset \Mtil$ the set of elementary tensors with tensor factors in positions $i \in \cF$ belonging to $W$. By density, it suffices to prove \eqref{eq.absurd-2} for all finite subsets $\cF \subset I$ and all $x,y \in W^\cF$. Since $h_n \cdot i \recht \infty$ for all $i \in I$, we have that $\cF \cap h_n \cdot \cF = \emptyset$ for all $n$ large enough. In what follows, we may thus assume that $\cF \cap h_n \cdot \cF = \emptyset$. If at least one of the factors of $x \in W^\cF$ belongs to $W_2$, using $\cF \cap h_n \cdot \cF = \emptyset$, one checks that
$$E_{M \ovt Q}((x^\star \ot 1) \theta_1(w) (\al_{h_n}(y) \ot 1)) = 0 \quad\text{for all}\;\; w \in M \ovt Q \; .$$
The same holds if one of the factors of $y \in W^\cF$ belongs to $W_2$. So to conclude the proof of \eqref{eq.absurd-2}, we may assume that $x = \theta_1(a) b$ and $y = \theta_1(c) d$ with $a,b,c,d \in M_0$. But then,
$$E_{M \ovt Q}((x^\star \ot 1) \theta_1(w_{h_n}) (\al_{h_n}(y) \ot 1)) = (b^\star \ot 1) \, E_{M \ovt Q}(\theta_1((a^\star \ot 1) \, w_{h_n} \, (\al_{h_n}(c) \ot 1))) \, (\al_{h_n}(d) \ot 1) \; .$$
Since
$$E_{M \ovt Q}(\theta_1((a^\star \ot 1) \, w_{h_n} \, (\al_{h_n}(c) \ot 1))) = 1 \ot (\vphi \ot \id)((a^\star \ot 1) \, w_{h_n} \, (\al_{h_n}(c) \ot 1)) \; ,$$
we conclude that \eqref{eq.absurd-2} follows from the assumption in \eqref{eq.absurd}.

By density, \eqref{eq.absurd-2} implies that
\begin{equation}\label{eq.absurd-3}
\|E_{M \ovt Q}(x^\star \theta_1(w_{h_n}) \al_{h_n}(y))\|_\vphi \recht 0 \quad\text{for all}\;\; x,y \in (\Mtil \ovt Q)_0 \; .
\end{equation}
Taking $x = y = V$ and using \eqref{eq.aim3}, it follows that $\|p_0\|_2 = \|p_0 w_{h_n}\|_2 \recht 0$. This is absurd and therefore, there is no sequence $h_n \in \Sigma$ satisfying \eqref{eq.absurd}.

{\bf Step 5.} We denote by $\Delta_{\vphi}$ the modular operator of $\vphi$ on $L^2(M)$ and by $\vphih$ the normal semifinite faithful weight $\Tr(\Delta_{\vphi} \, \cdot) \ot \vphi$ on $B(L^2(M)) \ovt Q$. Denote $\cN := (B(L^2(M)) \ovt Q)_{\vphih}$ and note that $\vphih$ is a semifinite trace on $\cN$. We denote by $a_h \in \cU(L^2(M))$ the automorphism $\al_h$ viewed as a unitary operator on $L^2(M)$. Using the formula $\Ad a_g \ot \al_g$, the automorphism $\al_g$ of $M \ovt Q$ is naturally extended to $B(L^2(M)) \ovt Q$. We still denote this extension by $\al_g$. We claim that there exists a nonzero projection $T \in p \cN p$ satisfying
\begin{equation}\label{eq.aim5}
\vphih(T) < \infty \quad\text{and}\quad T = w_h \, \al_h(T) \, w_h^\star \quad\text{for all}\;\; h \in \Sigma \; .
\end{equation}
Consider $\Sigma \actson \ell^2(I)$ given by $\Sigma \actson I$. Since there is no sequence $h_n \in \Sigma$ satisfying \eqref{eq.absurd}, there exists $\delta > 0$, a finite subset $\cF \subset I$ and finitely many elements $x_1,\ldots,x_m \in M$ that are eigenvectors for the modular group of $\vphi$ with eigenvalues $\lambda_1,\ldots,\lambda_m$, such that
\begin{equation}\label{eq.what-we-have}
\langle h \cdot 1_\cF,1_\cF \rangle + \sum_{i,j=1}^m \lambda_i \, \|(\vphi \ot \id)((x_j^\star \ot 1) w_h (\al_{h}(x_i) \ot 1))\|_\vphi^2 \geq \delta \quad\text{for all}\;\; h \in \Sigma \; .
\end{equation}
Whenever $x \in M$ is an eigenvector for the modular group of $\vphi$, we consider the rank one operator $T_x$ on $B(L^2(M))$ given by $T_x(y) = x \vphi(x^\star y)$. Note that $T_x \ot 1 \in \cN$. Define, in the Hilbert space $\cK = \ell^2(I) \oplus L^2(p \cN p , \vphih)$, the vector
$$T_0 = 1_\cF \oplus p\Bigl(\sum_{i=1}^m T_{x_i} \ot 1 \Bigr)p \; .$$
The formula $h \cdot (\xi,S) = (h \cdot \xi, w_h \al_h(S) w_h^\star)$ defines a unitary representation of $\Sigma$ on $\cK$. Formula \eqref{eq.what-we-have} says that
$$\langle h \cdot T_0, T_0 \rangle \geq \delta \quad\text{for all}\;\; h \in \Sigma \; .$$
Therefore, the unitary representation $\Sigma \actson \cK$ has a nonzero invariant vector. Since $\Sigma \actson \ell^2(I)$ has no nonzero invariant vectors, the representation $\Sigma \actson L^2(p \cN p,\vphih)$ has a nonzero invariant vector. Using functional calculus, it follows that \eqref{eq.aim5} holds.

{\bf Step 6.} There exists a nonzero projection $p_0 \in (M \ovt Q)_\vphi$ with $p_0 \leq p$, a Hilbert space $H$, a projection $q \in Q_\vphi \ovt B(H)$ with $(\vphi \ot \Tr)(q) <\infty$, a positive number $\lambda > 0$, an element $V \in (M \ovt Q)_{\vphi,\lambda} \ot \Hbar$ with $V V^\star = p_0$, $V^\star V = 1 \ot q$ and a generalized $1$-cocycle $(W_h)_{h \in \Sigma}$ for the action $(\al_h)_{h \in \Sigma}$ on $Q_\vphi \ovt B(H)$ with support projection $q$, such that
\begin{equation}\label{eq.aim6}
w_h \, \al_h(V) = V (1 \ot W_h) \quad\text{for all}\;\; h \in \Sigma \; .
\end{equation}

Denote by $P_\lambda$ the orthogonal projection of $L^2(M \ovt Q)$ onto $L^2((M \ovt Q)_{\vphi,\lambda})$. Take a nonzero projection $T \in p \cN p$ satisfying \eqref{eq.aim5}. Note that $\cN$ commutes with the projections $P_\lambda$. Choose $\lambda > 0$ such that $T P_\lambda$ is nonzero. By construction, $T P_\lambda$ is the orthogonal projection onto a closed subspace $\cE \subset L^2(p (M \ovt Q)_{\vphi,\lambda})$ that is a right $(1 \ot Q_\vphi)$-module and that is globally invariant under $\xi \mapsto w_h \, \al_h(\xi)$ for all $h \in \Sigma$.

We claim that $\cE$ has finite $(1 \ot Q_\vphi)$-dimension. Denote by $\cS$ the commutant of the right $(1 \ot Q_\vphi)$-action on $L^2((M \ovt Q)_{\vphi,\lambda})$. Note that $\cN P_\lambda \subset \cS$. The normal faithful tracial state $\vphi$ on $Q_\vphi$ induces the normal semifinite faithful trace $\tauh$ on $\cS$. To prove the claim, we must show that $\tauh(T P_\lambda) < \infty$. Denote by $E_\lambda \in \cZ(\cN)$ the smallest projection in $\cN$ that dominates $P_\lambda$. Applying \cref{lem.trace-formula} to $1 \ot Q \subset M \ovt Q$, we get that
\begin{equation}\label{eq.trace-formula}
\tauh(S P_\lambda) = \frac{1}{\lambda} \vphih(S E_\lambda) \quad\text{for all}\;\; S \in \cN^+ \; .
\end{equation}
In particular, the claim that $\tauh(T P_\lambda) < \infty$ follows from the property that $\vphih(T) < \infty$.

Choose a Hilbert space $H$, a nonzero projection $q \in Q_\vphi \ovt B(H)$ with $(\vphi \ot \Tr)(q) < \infty$ and a $Q_\vphi$-linear unitary operator $\Psi : q(L^2(Q_\vphi) \ot H) \recht \cE$. Fix $h \in \Sigma$. The unitary operator
$$\al_h(q) (L^2(Q_\vphi) \ot H) \recht q (L^2(Q_\vphi) \ot H) : \xi \mapsto \Psi^{-1}\bigl(w_h (\al_h \circ \Psi \circ \al_h^{-1})(\xi)\bigr)$$
is right $Q_\vphi$-linear. So it corresponds to left multiplication with $W_h \in Q_\vphi \ovt B(H)$ satisfying $W_h W_h^\star = q$ and $W_h^\star W_h = \al_h(q)$. By construction, $(W_h)$ is a generalized $1$-cocycle for the action $(\al_h)_{h \in \Sigma}$ with support projection $q$.

The operator $\Theta \in B(H, p L^2((M \ovt Q)_{\vphi,\lambda}))$ given by $\Theta(\xi) = \Psi(q(1 \ot \xi))$ satisfies
$$\Tr(\Theta^\star \Theta) = (\vphi \ot \Tr)(q) < \infty \; .$$
Therefore, $\Theta$ can be viewed as a nonzero vector $V \in p (L^2((M \ovt Q)_{\vphi,\lambda}) \ot \Hbar)(1 \ot q)$ satisfying
\begin{equation}\label{eq.cohom}
w_h \, \al_h(V) = V (1 \ot W_h) \quad\text{for all}\;\; h \in \Sigma \; .
\end{equation}
The left support of $V$ is a projection $p_0 \in (M \ovt Q)_\vphi$ satisfying $p_0 \leq p$. The right support of $V$ is a projection $q_1 \in (M \ovt Q)_\vphi \ovt B(H)$ satisfying $q_1 \leq 1 \ot q$. By \eqref{eq.cohom}, we also have that $\al_h(q_1) = (1 \ot W_h^\star) q_1 (1 \ot W_h)$ for all $h \in \Sigma$. \Cref{lem.weak-mixing}.2 implies that $q_1 = 1 \ot q_2$ for some projection $q_2 \in Q_\vphi \ovt B(H)$ satisfying $q_2 \leq q$. But then $\Psi(\xi) = 0$ for all $\xi$ in the image of $q-q_2$. Since $\Psi$ is unitary, we conclude that $q_2 = q$. Taking the polar part of $V$, we have found the partial isometry $V$ satisfying \eqref{eq.aim6}.

{\bf Step 7.} We upgrade step~6 from $\Sigma$ to the entire group $\Lambda$. More precisely, we prove that there exists a nonzero projection $p_0 \in (M \ovt Q)_\vphi$ with $p_0 \leq p$, a Hilbert space $H$, a projection $q \in Q_\vphi \ovt B(H)$ with $(\vphi \ot \Tr)(q) <\infty$, a positive number $\lambda > 0$, an element $V \in (M \ovt Q)_{\vphi,\lambda} \ot \Hbar$ with $V V^\star = p_0$, $V^\star V = 1 \ot q$ and a generalized $1$-cocycle $(W_g)_{g \in \Lambda}$ for the action $(\al_g)_{g \in \Lambda}$ on $Q_\vphi \ovt B(H)$ with support projection $q$, such that
\begin{equation}\label{eq.aim7}
w_g \, \al_g(V) = V (1 \ot W_g) \quad\text{for all}\;\; g \in \Lambda \; .
\end{equation}
Denote by $\beta_g$ the action of $\Lambda$ on $p(M \ovt Q)_\vphi p$ given by $\beta_g = (\Ad w_g) \circ \al_g$. Fix the positive number $\lambda > 0$ that appears in the statement of step~6. Denote by $\cP$ the set of projections $p_0 \in (M \ovt Q)_\vphi$ with $p_0 \leq p$ such that there exist $H$, $q$ and $W_h$ satisfying the conclusion of step~6. By \eqref{eq.aim6}, all $p_0 \in \cP$ satisfy $\beta_h(p_0) = p_0$ for all $h \in \Sigma$. We make the following four observations about $\cP$.

Take $g \in \Gamma$. Since $\Gamma$ and $\Sigma$ commute, we can apply $x \mapsto w_g \al_g(x)$ to \eqref{eq.aim6} and conclude that $\beta_g(p_0) \in \cP$ for all $p_0 \in \cP$.

We next prove that if $p_0 \in \cP$ and $p_1 \leq p_0$ is a projection satisfying $\beta_h(p_1) = p_1$ for all $h \in \Sigma$, then also $p_1 \in \cP$. This follows by multiplying \eqref{eq.aim6} on the left by $p_1$ and using \cref{lem.weak-mixing}.2 to conclude that the right support of $p_1 V$ belongs to $1 \ovt Q_\vphi \ovt B(H)$ and can be used to cut down $W_h$.

We next prove that if $p_1,p_2 \in \cP$, then also $p_0 = p_1 \vee p_2$ belongs to $\cP$. Taking the direct sum of the $V$, $H$, $q$ and $W_h$ that come with $p_1$ and $p_2$, we find an element $V \in (M \ovt Q)_{\vphi,\lambda} \ovt \Hbar$ and a generalized $1$-cocycle $W_h$ with support $q$ such that \eqref{eq.aim6} holds and such that the left support of $V$ equals $p_0$. As in the previous paragraph, the right support of $V$ belongs to $1 \ovt Q_\vphi \ovt B(H)$. Cutting down with this projection and replacing $V$ by its polar part, we find that $p_0 \in \cP$.

We finally make the obvious observation that $\cP$ is closed under taking a direct sum of projections that are orthogonal.

By step~6, the set $\cP$ is nonempty. Combining the four observations above, we find a nonzero projection $p_0 \in \cP$ that satisfies $\beta_g(p_0)=p_0$ for all $g \in \Gamma$. Take the corresponding $V$, $H$, $q$ and $W_h$. To prove that \eqref{eq.aim7} holds, it suffices to prove that $V^\star w_g \al_g(V)$ belongs to $1 \ovt Q_\vphi \ovt B(H)$ for all $g \in \Lambda$. Since we already know this when $g \in \Sigma$, it remains to consider $g \in \Gamma$. Put $X = V^\star w_g \al_g(V)$. Since $\Gamma$ and $\Sigma$ commute, and using \eqref{eq.aim6}, we find that
$$X = \kappa(g,h) \, (1 \ot W_h) \, \al_h(X) \, (1 \ot \al_g(W_h)^\star) \quad\text{for all}\;\; h \in \Sigma \; .$$
By \cref{lem.weak-mixing}.2, it follows that $X \in 1 \ovt Q_\vphi \ovt B(H)$. This concludes the proof of step~7.

{\bf End of the proof.} Having found $p_0$ as in step~7, we have $p_0 w_g = w_g \al_g(p_0)$ for all $g \in \Lambda$ and then $((p-p_0) w_g)_{g \in \Lambda}$ is a new generalized $1$-cocycle with support projection $p-p_0$. We can apply the reasoning above to $(p-p_0)w_g$. Therefore, the theorem follows by a maximality argument.
\end{proof}

\begin{corollary}\label{cor.special-case}
Let $\alpha : \Lambda \curvearrowright (M,\vphi) = (P,\phi)^I$ be a generalized Bernoulli action satisfying the same assumptions as in \cref{thm.cocycle-cs}. Let $w_g \in M_\vphi$ be a generalized $1$-cocycle with support projection $p \in M_\vphi$ and scalar $2$-cocycle $\Om$.

If the action $\Ad w_g \circ \al_g$ of $\Lambda$ on $p M_\vphi p$ has a trivial fixed point algebra, then there exists an integer $n \in \mathbb{N}$, an irreducible $\Om$-representation $\pi : \Lambda \recht \cU(n)$ and, writing $\lambda = \vphi(p) / n$, an element $V \in M_{\vphi,\lambda} \ot \overline{\C^n}$ satisfying $VV^\star = p$, $V^\star V = 1 \ot 1$ and
$$w_g = V (1 \ot \pi(g)) \al_g(V^\star) \quad\text{for all}\;\; g \in \Lambda \; .$$

If the action $\Ad w_g \circ \al_g$ of $\Lambda$ on $p M_\vphi p$ even has no nontrivial finite dimensional globally invariant subspaces, then $n=1$ and $\pi(g) \in \T$ satisfies $\pi(g) \pi(g') = \Om(g,g') \pi(gg')$ for all $g,g' \in \Lambda$.
\end{corollary}
\begin{proof}
By \cref{thm.cocycle-cs}, we find projections $p_k \in M_\vphi$ such that $p = \sum_k p_k$ and $p_k =  w_g \alpha_g(p_k) w_g^\star$, and we find Hilbert spaces $H_k$, positive numbers $\lambda_k > 0$, numbers $n_k \in \mathbb{N}_0$ and elements $v_k \in M_{\vphi,\lambda_k} \ovt \overline{\mathbb{C}^{n_k}}$ such that
\begin{align*}
 v_k v_k^\star = p_k \,\,\text{and}\,\,v_k^\star v_k = 1 \otimes 1, \quad v_k^\star w_g \alpha_g(v_k) = 1 \otimes W_{k,g}
\end{align*}
for $(W_{k,g})_{g \in \Lambda} \in M_{n_k}(\mathbb{C})$ unitaries satisfying
\begin{align*}
 W_{k,g} W_{k,g'} = \Omega(g,g') W_{k,gh} \quad\text{for }g,g' \in \Lambda.
\end{align*}
Assume that the action $\Ad w_g \circ \alpha_g$ on $p M_\varphi p$ has trivial fixed point algebra. Since all the $p_k$'s are invariant under $\Ad w_g \circ \alpha_g$, necessarily $p = p_1$. Because $v_1 \in M_{\varphi,\lambda_k} \otimes \overline{\mathbb{C}^{n_1}}$, we have $\varphi(p) = \varphi(v_1 v_1^\star) = \lambda_1 (\varphi \otimes \Tr)(v_1^\star v_1) = \lambda_1 n_1$. Putting $V = v_1, \lambda = \lambda_1, n = n_1$ and $\pi(g) = W_{1,g}$, we arrive at the conclusions of the corollary.

It is easy to see that $V  (1 \otimes M_n(\mathbb{C})) V^\star$ is a finite dimensional subspace of $p M_\varphi p$, globally invariant under $\Ad w_g \circ \alpha_g$. If no such subspace exists, then it follows that $n=1$, and $\pi(g) \in \T$.
\end{proof}

In the proof of \cref{thm.cocycle-cs}, we used the following general lemma. Assume that $(N,\vphi)$ is a von Neumann algebra with a faithful normal almost periodic state, that $Q \subset N$ is a von Neumann subalgebra and that $E : N \recht Q$ is a state preserving conditional expectation. Note that $\sigma_t^\vphi(Q) = Q$ and that the restriction of $\sigma_t^\vphi$ to $Q$ equals the modular automorphism group of $Q$.

Denote by $J$ and $\Delta$ the modular conjugation and modular operator on $L^2(N)$. Denote by $e : L^2(N) \recht L^2(Q)$ the orthogonal projection of $L^2(N)$ onto $L^2(Q)$. The von Neumann algebra $\langle N,e \rangle$ acting on $L^2(N)$ generated by $N$ and $e$ coincides with the commutant $J Q' J$ and is called the basic construction. It comes with a canonical normal semifinite faithful weight $\vphih$ characterized by the following two properties: $\sigma^{\vphih}_t = \Ad \Delta^{it}$ and $\vphih(x e y) = \vphi(xy)$ for all $x,y \in N$.

For every $\lambda > 0$, define the projection $E_\lambda \in \langle N,e \rangle$ as the join of the support projections of all $x e x^\star$, $x \in N_{\vphi,\lambda}$. Also, denote by $P_\lambda$ the orthogonal projection of $L^2(N)$ onto $L^2(N_{\vphi,\lambda})$. Define $\cS_\lambda$ as the commutant of the right action of $Q_\vphi$ on $L^2(N_{\vphi,\lambda})$. The restriction of $\vphi$ to $Q_\vphi$ is a normal faithful trace $\tau$ on $Q_\vphi$. This induces a normal semifinite faithful trace $\tauh$ on $\cS_\lambda$ characterized by the formula $\tauh(V W^\star) = \tau(W^\star V)$ whenever $V,W : L^2(Q_\vphi) \recht L^2(N_{\vphi,\lambda})$ are bounded, right $Q_\vphi$-linear operators.

\begin{lemma} \label{lem.trace-formula}
We have $\langle N,e \rangle_{\vphih} \, P_\lambda = \cS_\lambda$ and $\tauh(T P_\lambda) = \frac{1}{\lambda} \vphih(T E_\lambda)$ for all $T \in \langle N,e \rangle_{\vphih}^+$.
\end{lemma}
\begin{proof}
Note that $\langle N,e \rangle_{\vphih}$ is generated by $\{ a e b^\star \mid \mu > 0 \; , \; a,b \in N_{\vphi,\mu} \}$. Define $\cT_\lambda$ as the von Neumann algebra generated by $\{ x e y^\star \mid x,y \in N_{\vphi,\lambda}\}$. By definition, $E_\lambda$ is the unit of $\cT_\lambda$. When $\mu > 0$, $a,b \in N_{\vphi,\mu}$ and $x,y \in N_{\vphi,\lambda}$, we have
$$a e b^\star \; x e y^\star = a E(b^\star x) e y^\star \; .$$
It follows that $E_\lambda$ is a central projection in $\langle N,e \rangle_{\vphih}$ and $\langle N,e \rangle_{\vphih} \; E_\lambda = \cT_\lambda$.
For $x \in N_{\vphi,\lambda}$, define $V_x : L^2(Q_\vphi) \recht L^2(N_{\vphi,\lambda})$ by $V_x(a) = xa$. For all $x,y \in N_{\vphi,\lambda}$, we have $x e y^\star \; P_\lambda = V_x V_y^\star$.
It follows that $\cT_\lambda \, P_\lambda = \cS_\lambda$ and
$$\tauh(x e y^\star P_\lambda) = \tauh(V_x V_y^\star) = \tau(V_y^\star V_x) = \vphi(y^\star x) = \frac{1}{\lambda} \, \vphi(x y^\star) \; .$$
Since $\vphih(x e y^\star) = \vphi(xy^\star)$, the lemma is proved.
\end{proof}

\bibliography{free-group-bernoulli-shifts}{}
\bibliographystyle{hmalpha}
\end{document}